\documentclass[1p,review]{elsarticle}

\usepackage[breaklinks,colorlinks=true,bookmarks=false,citecolor=blue,urlcolor=blue]{hyperref}

\usepackage[utf8]{inputenc}
\usepackage[french,english]{babel}
\usepackage{algorithm}
\usepackage{algpseudocode}
\usepackage{amsfonts,amssymb}
\usepackage{amsthm}
\usepackage{amsmath}

\newtheorem{theorem}{Theorem}
\newtheorem{assumption}{Assumption}
\newtheorem{lemma}{Lemma}
\newtheorem{definition}{Definition}
\usepackage{subcaption}
\usepackage[dvipsnames]{xcolor}

\newcommand{\revise}[1]{{\color{red}#1}}
\newcommand{\hm}{iterative model combination algorithm}
\newcommand{\Hm}{Iterative model combination algorithm}
\usepackage{enumitem}

\usepackage{cleveref}
\crefname{figure}{figure}{figures}
\crefname{algorithm}{Algorithm}{Algorithms}
\crefname{theorem}{theorem}{theorems}
\crefname{assumption}{assumption}{assumptions}
\crefname{definition}{definition}{definitions}
\crefname{table}{table}{tables}
\crefname{equation}{equation}{equations}

\journal{Physica D: Nonlinear Phenomena}

\biboptions{sort&compress}

\begin{document}

\begin{frontmatter}

\title{Non-intrusive model combination \\
for learning dynamical systems}


\author[nus-math,cnrs]{Shiqi Wu}

\author[cnrs-fr,iuf]{Ludovic Chamoin}

\author[nus,cnrs]{Qianxiao Li\corref{ql_address}}
\cortext[ql_address]{Corresponding author}
\ead{qianxiao@nus.edu.sg}

\address[nus-math]{Department of Mathematics, National University of Singapore, 117543, Singapore}
\address[nus]{Department of Mathematics and Institute for Functional Intelligent Materials, National University of Singapore, 117543, Singapore}
\address[cnrs]{CNRS@CREATE LTD, 1 Create Way, CREATE Tower, 138602, Singapore}
\address[cnrs-fr]{\begin{otherlanguage}{french}
    Université Paris-Saclay, CentraleSupélec, ENS Paris-Saclay, CNRS, LMPS – Laboratoire de Mécanique Paris-Saclay, 91190 Gif-sur-Yvette, France
    \end{otherlanguage}
    }
\address[iuf]{\begin{otherlanguage}{french}
        Institut Universitaire de France, Paris, France
    \end{otherlanguage}
    }

\begin{abstract}
    In data-driven modelling of complex dynamic processes, it is often desirable to combine different classes of models to enhance performance.  Examples include coupled models of different fidelities, or hybrid models based on physical knowledge and data-driven strategies.  A key limitation of the broad adoption of model combination in applications is intrusiveness: training combined models typically requires significant modifications to the learning algorithm implementations, which may often be already well-developed and optimized for individual model spaces.  In this work, we propose an iterative, non-intrusive methodology to combine two model spaces to learn dynamics from data.  We show that this can be understood, at least in the linear setting, as finding the optimal solution in the direct sum of the two hypothesis spaces, while leveraging only the projection operators in each individual space.  Hence, the proposed algorithm can be viewed as iterative projections, for which we can obtain estimates on its convergence properties.  To highlight the extensive applicability of our framework, we conduct numerical experiments in various problem settings, with particular emphasis on various hybrid models based on the Koopman operator approach. This manuscript has been accepted by \href{https://www.sciencedirect.com/science/article/abs/pii/S0167278924001039}{\textit{Physica D: Nonlinear Phenomena}}.
\end{abstract}


\end{frontmatter}


\section{Introduction}\label{sec:a_intro}

Model combination plays a crucial role in learning dynamics by bridging
different approaches, effectively tackling difficulties like incomplete
physical models~\cite{wang2023coupled, reinhart2017hybrid, golemo2018sim,
mehta2021neural,mitusch2021hybrid}, multi-fidelity data~\cite{wackers2020multi},
and model stabilization~\cite{linot2023stabilized}.
The main idea behind model combination is to utilize domain knowledge to handle the
difficulties caused by nonlinearity and high-dimensionality in complex dynamical
systems.
There are two primary ways where model combination can be beneficial.
First, there may be partial physical models that can be leveraged,
and data-driven components model the corrections to these models
to improve prediction accuracy~\cite{willard2020integrating,von2020combining,von_Rueden_2021,karniadakis_physics-informed_2021,beckh2021explainable,mitusch2021hybrid}.
Second, general structural information on the underlying dynamics
may point to effective ways of model combination.
For example, the Koopman operator approach~\cite{brunton2021modern, tu2013dynamic,williams2015data}
can effectively model nonlinear dynamics, but
it requires the judicious choice of dictionary functions
that becomes challenging for high-dimensional systems.
However, for dynamics with a known structure such as reaction-diffusion equations,
this difficulty can be alleviated by combining a high-dimensional linear model (diffusion)
and a nonlinear point-wise Koopman model (reaction).
The latter only requires the approximation of 1D functions,
thus the choice of dictionary is much easier (see examples in~\cref{sec:experiments}).

While model combination is desirable, the current algorithmic approaches
suffer from two main limitations.
The first is sub-optimality, which manifests itself in a number of works
on ``residual learning'', which learns data-driven corrections to existing
theoretical models~\cite{reinhart2017hybrid, golemo2018sim, mehta2021neural}.
These methods first select a model in the known physical model space
(the first hypothesis space) and then learn the correction
using a data-driven method (the second hypothesis space).
We will show in~\cref{sec:meth} that in general, this approach does not
give the best possible combined model, even if each of the above
steps achieve minimal error individually in its fitting process.
An alternative to the two steps process is to train the two
models simultaneously.
Examples include coupled physics-informed neural networks~\cite{wang2023coupled,raissi_physics-informed_2019},
the combination of neural networks and the finite-element method~\cite{mitusch2021hybrid}, 
and the combination of linear and nonlinear neural networks for
model stabilization~\cite{linot2023stabilized}.
Although these joint-training processes ensure better optimality,
they bring about the second limitation of intrusiveness
as implementing them typically requires significant
modification to the underlying code.
Non-intrusiveness is a highly desirable algorithmic property
in the analysis of dynamical systems.
For instance, in the context of model reduction,
many studies employ neural networks to
perform non-intrusive dimensionality reduction
on commercial simulation software~\cite{bai2021non, hesthaven2018non}.
Additionally, an intrusive approach to
combination can give rise to additional challenges.
For instance, certain data-driven methods have well-developed optimization techniques
(e.g., we can achieve fast convergence when solving linear regression problems using the GMRES method~\cite{saad2003iterative}).
Intrusively combining these two kinds of
data-driven methods may result in the loss of the favorable properties
associated with optimization.  Also, the combination of two simple
models could lead to a complex problem to be jointly solved. An example of this
complexity is evident in robust principal component
analysis~\cite{cai2019accelerated}, where the objective is to find a
representation of the target matrix as a combination of one sparse matrix and
one low-rank matrix.  Taking into consideration the limitation above, proposing
a non-intrusive model combination approach is a significant challenge in
learning dynamics.

To tackle these challenges, in this work we propose an efficient framework,
referred to as the \emph{\hm}, which provides a
non-intrusive approach to combining two data-driven models
through additive combination.
Moreover, unlike residual learning, our approach is provably optimal
in the linear setting when both hypothesis spaces are closed subspaces of a
Hilbert space.
In this case, the algorithm is equivalent to
alternatively projecting the residual vector onto
the orthogonal complements of the two hypothesis spaces respectively.
We prove that our algorithm converges linearly, with explicit a
priori and a posteriori estimates for the trajectory error.
The latter can be
used as a principled stopping criterion.  Additionally, we introduce an
accelerated version of this algorithm, which is highly efficient for specific
classes of problems where one model hypothesis space is low-dimensional.
We validate these methods across various problem settings, including
diffusion-reaction equations, forced PDE-ODE coupled models in cardiac
electrophysiology, and parameterized tail-accurate control systems,
highlighting the extensive applicability of our framework.
We pay particular attention to Koopman operator-based methods~\cite{brunton2021modern,li2017extended,guo2023learning,korda2018linear}
as candidate hypothesis spaces for data-driven methods,
due to their wide applicability in learning nonlinear dynamics.
Following this introduction, in~\cref{sec:prob} we introduce
the problem formulation of non-intrusive model combination in learning
dynamical systems with a simple example of reaction-diffusion equation.
In~\cref{sec:meth}, we propose
a linearly convergent \hm \ and its acceleration scheme.
We then provide theoretical
analysis and experimental evidence validating the convergence patterns.
\Cref{sec:experiments} showcases numerical results
on three different problem settings and choices of model hypothesis spaces, and discusses  the advantages and limitations of the proposed method. Eventually, conclusions are prospects are drawn in~\cref{sec:conclusion}.
The implementation of our method and the reproduction of presented experiments
are found in~\cite{shiqi2023.Nonintrusivemodelcombinationforlearningdynamics}.


\section{Problem formulation}
\label{sec:prob}
\subsection{Learning dynamics by data-driven methods}
Let us consider
a dynamical system in state space $X\subset \mathbb{R}^K$
with a high-dimensional, nonlinear evolution function
$F:X\to X$.
The states $x_k\in X$ satisfy $x_{k+1} = F(x_k)$ for $k =
0,1,2,3,\ldots,N$.
Our objective is to learn the evolution function $F$
using data-driven methods
and then address both predictive and control problems
for this dynamical system.
That is, we aim to find a model $\tilde{F}$ such that
\begin{equation}
    \tilde{F} = \mathrm{argmin}_{\hat{F}\in \mathcal{G}}\mathbb{E}[L(F(x), \hat{F}(x))].
    \label{equ:0_risk}
\end{equation}
In this context, $\mathcal{G}$ is the hypothesis space of a selected data-driven model class,
$\mathbb{E}$ represents the expectation over the distribution of states $x$,
and $L(y, \hat{y})$ is the loss function which measures how different the prediction of a hypothesis $\hat{F}$ is from that of the true outcome $F$. Here, we choose the loss function $L(y, \hat{y})$ as mean square error
\begin{equation}
    L(y, \hat{y}) = \|y - \hat{y}\|^2.
\end{equation}
For illustration, in this paper we abstract the fitting or learning process on data generated by a ground
truth dynamics driven by $F$ (i.e, solving the minimization problem~\eqref{equ:0_risk}), using a hypothesis space $\mathcal{G}$
as a projection operator $P_{\mathcal{G}}$ 
\begin{equation}
    \tilde{F} = P_{\mathcal{G}}(F) = \mathrm{argmin}_{\hat{F}\in \mathcal{G}}\|F-\hat{F}\|.
    \label{equ:0_projection_notation}
\end{equation}
which means the optimal solution of minimization problem~\eqref{equ:0_risk} is a projection of $F$ onto the hypothesis space $\mathcal{G}$.

Suppose we have some snapshots ${x_0^j, x_1^j, \ldots, x_k^j, \ldots, x_m^j}$ from a set of trajectories generated from initial values $x_0^j,\ j = 1, 2, 3, \ldots, s$.
We can rearrange these snapshots into a set of pairs $\mathcal{D} = [(x_{k}^j, x_{k+1}^j)],\ k = 0, 1, 2, \ldots, m-1,\ j = 1, 2, 3, \ldots, s$.
Since calculating the expectation over the entire state set $X$ is challenging,
we typically solve the optimization problem based on empirical risk:
\begin{equation}
    \tilde{F} = \mathrm{argmin}_{\hat{F}\in G} \sum_{(x,y)\in \mathcal{D}}L(y, \hat{F}(x)).
    \label{equ:0_empirical}
\end{equation}
For convenience, in presenting our method and analysis,
we do not distinguish between the solutions of the optimization problems in~\eqref{equ:0_risk} and~\eqref{equ:0_empirical}.

Numerous machine learning techniques can be employed as the described projection. For example, linear regression involves projecting onto a linear subspace defined by explanatory variables. The Koopman operator approximates the evolution function projected onto a linear subspace spanned by the dictionary observables~\cite{williams2015data}.
Similarly in the nonlinear case, neural networks~\cite{metryar2012foundation} act as solvers for projecting
onto a nonlinear composition function space.

\begin{figure}
    \centering
    \includegraphics[width = 0.6\textwidth]{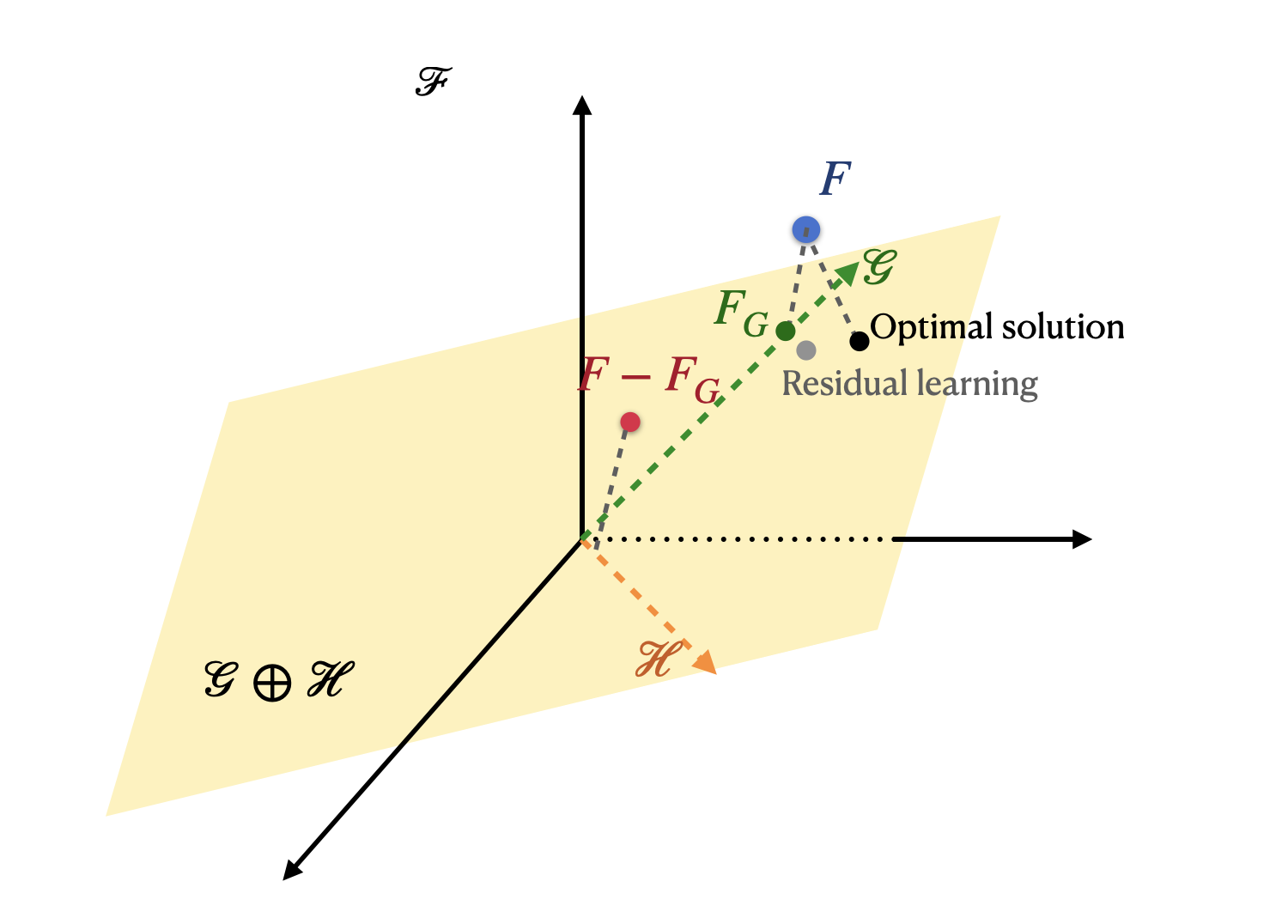}
    \caption{A cartoon for model combination. Suppose we aim to understand the evolution function $F$ within a broad, high-dimensional space denoted as $\mathcal{F}$.
    We can break down the model into two lower-dimensional components, $\mathcal{G}$ and $\mathcal{H}$, and subsequently learn $F$ within a direct-sum subspace. 
    $P_{\mathcal{G}\oplus \mathcal{H}}(F)$ is plotted as the black point in the yellow plane, which represents the optimal approximation. 
    However, if we conduct the residual learning process, projecting $F$ to $\mathcal{G}$ to find $F_\mathcal{G}$ (the green point) and then project the residual $F-F_\mathcal{G}$ (the red point) onto $\mathcal{H}$, 
    the final solution is plotted as the grey point, which is visually not the optimal approximation.
    }\label{fig:0_illustration}
\end{figure}

\subsection{An example: reaction-diffusion equation}
\label{sec:illustrate_example}

Purely data-driven methods may face challenges in learning dynamics due to high
dimensionality and complexity of dynamical systems.
However, combining models can effectively
solve this problem.
High-dimensional problems can often be decomposed into a number of
much simpler problems using a divide and conquer strategy and modelling different parts of dynamics with different data-driven methods.
Additionally, if some structural
information about the dynamics is already known, combining models can help
incorporate this prior knowledge to enhance prediction or
control performance.

Here, we illustrate the benefits of model
combination with a simple example of reaction-diffusion equation:
\begin{equation}
    \frac{du}{dt} = \mu \frac{d^2u}{dz^2} + R_\eta(u).
    \label{equ:0_reaction_diffusion}
\end{equation}
In this equation, $u(z, t)$ is a one-dimensional variable, $z\in [0,1]$ and $t\in [0,0.05]$. $u(z, 0)$ is randomly selected following uniform distribution in $[0,1]$
after discretization, and $u(0, t) = u(1,t)=0$. We observe
that~\eqref{equ:0_reaction_diffusion} can be divided into two distinct
components: the diffusion term $u(z, t) \mapsto \mu\frac{d^2 u}{dz^2}(z, t)$ represents a
linear model which is related to the spatial complexity of the discretization,
while the reaction term $u(z, t) \mapsto R_\eta(u(z, t))$ is a nonlinear and
pointwise function that depends on a parameter $\eta$, but
remains independent of the spatial dimension.
Considering this information, we can model these two
components separately and add them together.  One involves learning the
diffusion term as a linear model that relates $u(z_i, t_{j+1})$ to $u(z_i, t_j)$
and its neighboring values.
The other part involves learning the reaction term as a nonlinear problem based
on $u(z_i, t_j)$ for any $z_i$ within the discretization mesh.
This combination can decompose the high-dimensional problem into simpler
components while preserving the structure of the reaction-diffusion equation.
An example of application of this scheme
is the data-driven modelling of intrinsic self-healing
materials~\cite{anwarali2022dynamic}.
Detailed numerical results for learning the reaction-diffusion equation
are presented in~\cref{sec:example1}.
Furthermore, if we consider a control problem involving the adjustment of $\mu$ in~\eqref{equ:0_reaction_diffusion},
we can model the equation as two parts: $u_{n+1} = F_1(u_n, \mu) + F_2(u_n, \eta)$, and maintain the convexity of the optimization problem through specific structures for $F_1$ and $F_2$.
A numerical example of this type is shown in~\cref{sec:example3}.

\subsection{General problem formulation of model combination}
With the previous motivation in mind,
let us now introduce the general formulation of model combination.
We start with the direct-sum space:
for two subsets $\mathcal{G}$ and $\mathcal{H}$ of a Hilbert space,
the direct-sum space of $\mathcal{G}$ and $\mathcal{H}$ is defined as:
\begin{equation}
    \mathcal{G}\oplus \mathcal{H} = \{F_\mathcal{G} + F_\mathcal{H}:F_\mathcal{G}\in \mathcal{G}, F_\mathcal{H}\in \mathcal{H}\}.
\end{equation}
Suppose we have two methods that can respectively obtain the projection in the hypothesis spaces $\mathcal{G}$ and $\mathcal{H}$, and the snapshots of the dynamical system $(X, F, x_0)$ are known.
In the combination of two models, we aim to approximate the evolution function $F$ in the direct-sum space of $\mathcal{G}$ and $\mathcal{H}$, i.e.,
\begin{equation}
    \begin{split}
    \tilde{F} 
    & = \mathrm{argmin}_{\hat{F}\in \mathcal{G}\oplus\mathcal{H}}\|F-\hat{F}\|\\
     & = \tilde{F}_\mathcal{G} + \tilde{F}_\mathcal{H},\\
     \mathrm{where}\quad \tilde{F}_\mathcal{G}, \tilde{F}_\mathcal{H} &= \mathrm{argmin}_{\hat{F}_\mathcal{G}\in \mathcal{G}, \hat{F}_\mathcal{H}\in \mathcal{H}}\|F-(\hat{F}_\mathcal{G} + \hat{F}_\mathcal{H})\|.\\
    \end{split}
    \label{equ:0_formulation}
\end{equation}
The notation of projection is described in~\eqref{equ:0_projection_notation} and the hypothesis spaces $\mathcal{G}, \mathcal{H}$ are subsets in a Hilbert space.

According to the discussion of non-intrusiveness in~\cref{sec:a_intro},
we assume that we only have access to the
projection operators for each individual hypothesis space.
This can be legacy codes or highly optimized available routines.
A key challenge in addressing the optimization problem
is to find the optimal solution $\tilde{F}_\mathcal{G} + \tilde{F}_\mathcal{H}$
in a non-intrusive manner.
In other words, we only have access to $P_\mathcal{G}$ and $P_\mathcal{H}$, and do not have access to the projection operator
onto $\mathcal{G} \oplus \mathcal{H}$,
corresponding to the learning algorithms already developed for each hypothesis space.
The formulation of non-intrusive model combination is thus:
how can we obtain an approximation $\tilde{F}$ in~\eqref{equ:0_formulation}
using only the individual learning algorithms $P_\mathcal{G}$ and $P_\mathcal{H}$?
In \cref{sec:meth}, we propose an algorithm to address this problem.

\section{Methodology}
\label{sec:meth}
\subsection{
    Non-optimality of the residual learning
    }

Since solving for the optimal solution jointly might not be possible in a
non-intrusive manner, the goal is to identify the corresponding $F_\mathcal{G}\in
\mathcal{G}$ and $F_\mathcal{H}\in \mathcal{H}$ through the individual models. One
approach to address this problem, referred to as residual
learning~\cite{reinhart2017hybrid, golemo2018sim},  involves obtaining an
initial $F_\mathcal{G}$ within the hypothesis space $\mathcal{G}$ and subsequently,
employing a correction model with hypothesis space $\mathcal{H}$ to learn and refine the difference
$F - F_\mathcal{G}$, leading to the projection $P_{\mathcal{H}}(F - F_\mathcal{G})$. 
This leads to an integrated approximation denoted as $F \approx F_\mathcal{G} +
P_\mathcal{H}(F - F_\mathcal{G})$. However, the selection of an appropriate $F_\mathcal{G}$ could be
a challenging task. One potential strategy is to utilize a fitting process to
choose the approximation $F_\mathcal{G} = P_{\mathcal{G}}(F)$ of $F$ within $\mathcal{G}$.

Nonetheless, this residual learning process may not yield the closest approximation, even in simple problem settings.
 We visually illustrate the failure of residual learning using~\cref{fig:0_illustration}.
Here, we first project $F$ (the blue point) onto the subspace $\mathcal{G}$,
and then learn a correction by projecting $F-F_\mathcal{G}$ (the red point) onto the subspace $\mathcal{H}$.
We can immediately observe that the residual learning procedure does not give
the best approximation in the direct sum space.
We will show later in \cref{sec:example1} that for typical problems, residual
learning is non-optimal.

\subsection{The \hm}

By observing~\cref{fig:0_illustration},
a logical progression of residual learning
would be to cyclically perform the correction process.
Starting from an initial guess $F_\mathcal{H}^{0}$ (often taken to be 0),
we iteratively perform the following learning steps
\begin{equation}
    \begin{split}
        F_\mathcal{G}^n \gets P_\mathcal{G}(F-F_\mathcal{H}^n),\\
        F_\mathcal{H}^{n+1} \gets P_\mathcal{H}(F-F_\mathcal{G}^n),
    \end{split}
\end{equation}
until some stopping criterion is met.
It is worth noting that residual learning only conducts this iterative step
once, which is not theoretically optimal.
In contrast,
we provide in \cref{sec:converge} a
theoretical analysis of our method in the special case
where $\mathcal{G},\mathcal{H}$ are closed subspaces.
In particular, we show that the method exhibits a linear convergence pattern.
Furthermore, the convergence rate is related to the correlation
between the two hypothesis spaces, a notion that can be made
precise in the linear setting.

We summarize the proposed method in~\cref{alg:framework}.
The stopping criterion used in our numerical experiments is
\begin{equation}
    \label{equ:stopcriterion1}
    \frac{1}{N}\sum_{i=0}^N\|y^n_i -y_i\| < \epsilon.
\end{equation}
Here, $y_i^n$ represents the one-step prediction generated from the initial
condition $x_i$ by $F^n = F^n_{\mathcal{G}} + F^n_{\mathcal{H}}$, and $y_i$ comes from the real trajectory data,
where $y_i = F(x_i)$. Other stopping criteria can also be used, as described in
the a priori and posteriori estimates provided in \cref{sec:converge}, such as
\begin{equation}
    \label{equ:stopcriterion2}
    \frac{1}{N}\sum_{i=0}^N\|y^{n}_{i} - y^{n-1}_{i}\|<\epsilon.
\end{equation}

\begin{algorithm}
    \caption{The \hm}\label{alg:framework}
    \begin{algorithmic}
    \Require Data-driven models functioning as projection operators $P_{\mathcal{G}}$ and $P_{\mathcal{H}}$
    \State$n\gets 0$
    \State$F_\mathcal{H}^n\gets P_{\mathcal{H}}(F)$  \Comment{Get the best approximation of $F$ by $P_{\mathcal{H}}$}
    \While{stopping criterion is not satisfied}
    \State$F_\mathcal{G}^n\gets P_{\mathcal{G}}(F-F_\mathcal{H}^n)$    \Comment{Get the best approximation of $F-F_\mathcal{H}$ by $P_{\mathcal{G}}$}
    \State$F_\mathcal{H}^{n+1}\gets P_{\mathcal{H}}(F-F_\mathcal{G}^n)$    \Comment{Get the best approximation of $F-F_\mathcal{G}$ by $P_{\mathcal{H}}$}
    \State$n\gets n+1$
    \EndWhile
    \State$\tilde{F} = F_\mathcal{H}^n + F_\mathcal{G}^n$    \Comment{Obtain the best approximation in $\mathcal{H}\oplus \mathcal{G}$}
    \end{algorithmic}
\end{algorithm}

\subsection{Convergence analysis}
\label{sec:converge}
In this part, we give a theoretical analysis on convergence of
\hm\ (\cref{alg:framework})
in the subspace case.
The key idea in our analysis is the observation that
our algorithm is essentially an iterative projection
of the \emph{residual} $F - F_\mathcal{H}^{0}$
onto the orthogonal complements
$\mathcal{G}^\perp$ of $\mathcal{G}$
and $\mathcal{H}^\perp$ of $\mathcal{H}$ respectively.
Hence, the convergence and its rate
follow from the alternative projection method
proposed by John von Neumann~\cite{vonNeumann1949} to
find the intersection of two closed subspaces.
This projection
algorithm was later extended to convex sets by
Dykstra~\cite{dykstra_algorithm_1983} and applied to
multivariate regression problems~\cite{collatz2006approximation} and robust principal component analysis~\cite{cai2019accelerated}.

After presenting the convergence of the approximations ${F}^{n} := F^{n}_{\mathcal{G}} + F^{n}_{\mathcal{H}}$,
we provide estimates for the $k$-step predictive error after $n$ iterations,
denoted as $e_k^n = \|x_k^n - x_k\|$.
Estimates of the limiting value of this error, denoted as $e_k = \lim_{n\to+\infty}e_k^n$, are also presented.
Here, $x_k^n$ represents the predicted state at time $k$ based on the $n$-th iteration, while ${x_k}$ represents the actual trajectory data.
For all the theorems discussed in this section, we assume that
the dynamical system and models have the following properties in \cref{ass:1}.
\begin{assumption}
    \label{ass:1}
    \quad
    \begin{enumerate}[label = (\arabic*)]
        \item $\mathcal{H}$ and $\mathcal{G}$ are closed
        linear subspaces of a Hilbert space.
        \item $F^n_\mathcal{G}$ and $F^n_\mathcal{H}$ are iteratively computed by \cref{alg:framework}.
        \item The set of states $X\subset \mathbb{R}^K$ is bounded, and we denote $M = \max_{x\in X}\|x\|$.
        \item For illustrative purposes, we assume $\|F\|=1$.
    \end{enumerate}
\end{assumption}

We will now provide an analysis of the convergence of \cref{alg:framework}.

\begin{theorem}[Convergence]
\label{thm:converge}
    Under \cref{ass:1}, the following holds:
    \begin{enumerate}[label = (\arabic*)]
        \item
            $\lim_{n\to + \infty}\|P_{\mathcal{G}\oplus \mathcal{H}}(F) - F^n\| = 0$.

        \item As $n\to \infty$,
        the predicted trajectory $\{x^n_k\}$ converges to
        some $\{\tilde{x}_k\}$.
        Suppose the distance between $F$ and $\mathcal{G}\oplus \mathcal{H}$ equals to $\epsilon_F$. For any finite $k\in \mathbb{Z}^+$, the following inequality holds:
        \begin{equation}
        e_k = \|\tilde{x}_k - x_k\| \leq ((1 + \epsilon_F)^k - 1)M.
        \end{equation}
        In particular, when $\epsilon_F = 0$,
        \begin{equation}
            \lim_{n\to +\infty}e^n_k = 0,\quad \forall k\in \mathbb{Z}^+.
        \end{equation}
    \end{enumerate}
\end{theorem}
\Cref{thm:converge} demonstrates that as $n\to +\infty$, the evolution
function obtained through \cref{alg:framework} converges to the projection of
$F$ onto the direct-sum space $\mathcal{G}\oplus \mathcal{H}$.
Additionally, the predicted trajectory also converges.
The distance between the converged prediction
and the reference trajectory can be controlled, and the distance is 0
when $F$ is in $\mathcal{G}\oplus \mathcal{H}$. The detailed proof is shown in \ref{appendix1}.

To further analyze the estimate of the convergence rate, we introduce the
definition of minimum angle between subspaces~\cite{jordan1875essai}.
This concept provides a mathematical measure that will
help understand the relationship between the
convergence rate and the chosen models spaces.

\begin{definition}[Minimum angle between subspaces]
\label{def1}
    We define $\theta_0(\mathcal{G},\mathcal{H})$ as the minimum angle between subspaces. $c_0(\mathcal{G},\mathcal{H}) = \cos(\theta_0(\mathcal{G},\mathcal{H}))$ is given by
    \begin{equation}
        \sup\{|\langle g, h\rangle|:g\in \mathcal{G}, \|g\|\leq1, h\in \mathcal{H},\|h\|\leq 1\}.
    \end{equation}
    Define $c(\mathcal{G},\mathcal{H})=c_0(\mathcal{G}\cap (\mathcal{G}\cap \mathcal{H})^{\bot}, \mathcal{H}\cap (\mathcal{G}\cap \mathcal{H})^\bot)$, then $0\leq c(\mathcal{G},\mathcal{H})\leq c_0(\mathcal{G},\mathcal{H})\leq 1$, where $\mathcal{G}^\bot$, $\mathcal{H}^\bot$ represent the orthogonal complements of $\mathcal{G}$, $\mathcal{H}$.
\end{definition}

\begin{theorem}[Convergence rate]
    \label{thm:convrate}
    Under~\cref{ass:1}, the following conclusions hold:
    \begin{enumerate}[label = (\arabic*)]
    \item For all $n\in \mathbb{Z}^+$, $\|F^n-P_{\mathcal{G}\oplus \mathcal{H}}(F)\|\leq c(\mathcal{G},\mathcal{H})^{2n-1}$.
\item For any finite $k\in \mathbb{Z}^+, n\in \mathbb{Z}^+$, the following inequality holds:
\begin{equation}
        e_k^n = \|x^n_{k}-x_k\| \leq [(1+c(\mathcal{G},\mathcal{H})^{2n-1} + \epsilon_F)^k-1]M.\\
\end{equation}
\item Specifically, if $F\in \mathcal{G}\oplus\mathcal{H}$, then
\begin{equation}
    e_k^n = \|x^n_{k}-x_k\| \leq [(1+\|F - F^n\|)^k-1]M
\end{equation}
\begin{equation}
        e_k^n = \|x^n_{k}-x_k\| \leq [(1+\frac{c(\mathcal{G}, \mathcal{H})}{1-c^2(\mathcal{G}, \mathcal{H})}\|F^n - F^{n-1}\|)^k-1]M
\end{equation}
Furthermore, when $k\|F - F^n\|\leq 1,\ k\frac{c^2(\mathcal{G}, \mathcal{H})}{1-c(\mathcal{G}, \mathcal{H})}\|F^n - F^{n-1}\|<1$, we have
\begin{equation}
    e_k^n \leq k^2 M \|F - F^n\|.
\end{equation}
\begin{equation}
    e_k^n \leq k^2 M \frac{c(\mathcal{G}, \mathcal{H})}{1-c^2(\mathcal{G}, \mathcal{H})} \|F^n - F^{n-1}\|.
\end{equation}
\end{enumerate}
\end{theorem}

These inequalities indicate that the error
between $F^n$ and $P_{\mathcal{G}\oplus \mathcal{H}}(F)$
decreases exponentially with the number of iterations,
with a rate determined by the value of $c(\mathcal{G},\mathcal{H})$.
Also, when $F\in \mathcal{G}\oplus\mathcal{H}$,
the a posteriori estimate of \cref{alg:framework} presents clear evidence that
the prediction error $e_k^n$ is predominantly governed by the absolute difference between two successive estimated functions,
namely $\|F^n-F^{n-1}\|$. The proof of \cref{thm:convrate} is shown in \ref{appendix2}.
According to the estimates, we could establish a suitable stopping criterion as~\eqref{equ:stopcriterion1} and~\eqref{equ:stopcriterion2}.

Additionally, \cref{thm:convrate} could serve as a criterion for method design.
One may observe that $c(\mathcal{G}, \mathcal{H})$
measures the `orthogonality' between $\mathcal{G}$ and $\mathcal{H}$.
In other words, if the two models describe independent features of the problem, the algorithm will converge rapidly.
Otherwise, when the correlation between the two models increases, the convergence speed may be slower.
Thus, if one hypothesis space $\mathcal{G}$ is chosen, we can solve an optimization problem
over the parameterized set of hypothesis spaces $\mathcal{H}_\nu$ with the hyper-parameter \revise{$\nu$}
\begin{equation}
    \label{equ:0_ChoosingModel}
    \min_{\nu}
    c (\mathcal{G}, \mathcal{H_\nu})
\end{equation}
to find a particular hypothesis space $\mathcal{H}_{\nu^*}$ which
maximizes performance.

We now validate these theoretical conclusions with a simple numerical example.

\paragraph{Numerical verification}
In order to evaluate the convergence pattern in theoretical analysis,
we consider a linear PDE in the form of a diffusion equation
\begin{equation}
    \label{equ:1_diffusion}
    \frac{du}{dt} = \mu_1\frac{d^2 u}{dz_1^2} + \mu_2\frac{d^2 u}{dz_2^2}.
\end{equation}
In this equation, $u(z, t)$ is a one-dimensional variable with $z\in \Omega =
[0,1]^2$ and $t\in [0, 0.05]$. The boundary condition is
$u(z, t)$ equals to 0 for all $z$ in the boundary $\partial \Omega$ (i.e., homogeneous Dirichlet boundary conditions). When generating the data, we set $\mu_1 = \mu_2 = 1$ and select the initial value $u(z, 0)$
 from a uniform distribution $U(0,1)$ independently for each $z$.

To validate the relationship between the correlation of two hypothesis spaces and their convergence rate, we fix one of the hypothesis spaces and employ a hyper-parameter $\nu$ to control the degree of correlation in the other space. We choose two parametric hypothesis spaces as follows:
\begin{equation}
    \begin{split}
        \mathcal{G}: F_\mathcal{G}(u_{ij}^n;\lambda_1) = \lambda_1 \delta^2_{z_1} u^n_{ij},\\
        \mathcal{H}: F_\mathcal{H}(u_{ij}^n;\lambda_2) = \lambda_2 (\nu\delta^2_{z_1} u^n_{ij} + \delta^2_{z_2} u^n_{ij}),\\
    \end{split}
\end{equation}
where $\lambda_1, \lambda_2$ are the trainable parameters.
Here, $\delta_{z_1}^2 u,\delta_{z_2}^2 u$
represent the discretization of $\frac{d^2
u}{dz_1^2},\frac{d^2 u}{dz_2^2}$, as
\begin{equation}
    \label{equ: delta_z}
    \begin{split}
        \delta_{z_1}^2 u^n_{ij} = \frac{u^n_{i-1,j}-2u^n_{ij}+u^n_{i+1,j}}{(\Delta {z_1})^2},\\
        \delta_{z_2}^2 u^n_{ij} = \frac{u^n_{i,j-1}-2u^n_{ij}+u^n_{i,j+1}}{(\Delta {z_2})^2}.\\
    \end{split}
\end{equation}
$\Delta z_1$ and $\Delta z_2$ are chosen as $0.1$ in our experiments.

The hyper-parameter \(\nu\) can control the degree of correlation between the two models. 
When the value of \(\nu\) is changed, the correlation between the two models changes correspondingly. 
As the absolute value of \(\nu\) approaches infinity, the hypothesis space \(\mathcal{H}\) of the second model becomes increasingly similar to \(\mathcal{G}\), resulting in a significant correction. 
With varying values of the hyper-parameter \(\nu\), we can obtain corresponding \(\mathcal{H}_\nu\) with different correlations to the fixed model space \(\mathcal{G}\). 
We aim to experimentally validate the relationship between the correlation of these two hypothesis spaces and their convergence rate, as outlined in \cref{thm:convrate}.

By taking the logarithm on both sides of the inequality~\eqref{equ:A_27}, we can obtain
\begin{equation}
    \log(\|F^n-P_{\mathcal{G}\oplus\mathcal{H}}(F)\|)\leq (2n-1)c(\mathcal{G},\mathcal{H}) +\log(\|F\|).
\end{equation}

To validate our theoretical analysis, we conduct experiments with different values of $\nu$ (ranging from -4 to 4).
From~\cref{fig:1_ChoosingModel}, we observe that as the correlation between the two models increases,
the convergence speed becomes significantly slower. This indicates the crucial impact of `orthogonality' on the convergence behavior of our \hm.
Additionally, we used the BFGS method to solve the optimization problem~\eqref{equ:0_ChoosingModel} and obtained the optimal value of $\nu$, represented as $\tilde{\nu} = -0.6455$.
As depicted in \cref{fig:1_ChoosingModel}, \cref{alg:framework} converges directly to the optimal solution,
highlighting the effectiveness of adaptive model selection.

\Cref{fig:1_ChoosingModel_quant} illustrates the convergence rate of \cref{alg:framework} across varying values of the orthogonality parameter $\nu$,
with $n$ denoting the iteration steps.
The two models could be written as inner products of vectors:
\begin{equation}
    \begin{split}
        F_\mathcal{G}(u_{ij}^n;\lambda_1) = \lambda_1\mathbf{b_1}^T  \mathbf{u}\\
        F_\mathcal{H}(u_{ij}^n;\lambda_2) = \lambda_2(\nu\mathbf{b_1} + \mathbf{b_2})^T \mathbf{u}
    \end{split}
\end{equation}
Here, $\mathbf{b_1} = [1, 1,-2,0,0]^T,~\mathbf{b_2} = [0,0,-2,1,1]^T$ and $\mathbf{u} = [u_{i-1,j}^n, u_{i+1,j}^n, u_{i,j}^n, u_{i,j-1}^n, u_{i,j+1}^n]^T$.
We can calculate $c(\mathcal{G},\mathcal{H})$ by~\eqref{equ:1_slope}:
\begin{equation}
    c(\mathcal{G},\mathcal{H}) = \frac{\langle \mathbf{b_1}, \nu\mathbf{b_1} + \mathbf{b_2}\rangle}{\|\mathbf{b_1}\|\|\nu\mathbf{b_1} + \mathbf{b_2}\|} = \sqrt{1 - \frac{5}{9\nu^2+12\nu+9}}.
    \label{equ:1_slope}
\end{equation}
We could easily determine that the optimal value of $\nu$ is $-2/3$ in~\eqref{equ:1_slope}. This value closely aligns with the optimal value $\tilde{\nu}$ obtained through numerical methods.

In~\cref{fig:1_ChoosingModel_quant}, we evaluate both the reference and
experimental convergence rates for varying $\nu$. The reference rate $c(\mathcal{G}, \mathcal{H})$ is computed
using~\eqref{equ:1_slope}, while the experimental rates are determined as the
slopes of $\log(\|F^n-P_{\mathcal{G}\oplus \mathcal{H}}(F)\|)$ relative to $2n-1$.
Comparing our experimental rates with the reference values, we observe a close
match.  These results quantitatively support our theoretical analysis.

\begin{figure}[H]
    \centering
    \begin{subfigure}[b]{0.45\textwidth}
        \centering
        \includegraphics[width = \textwidth]{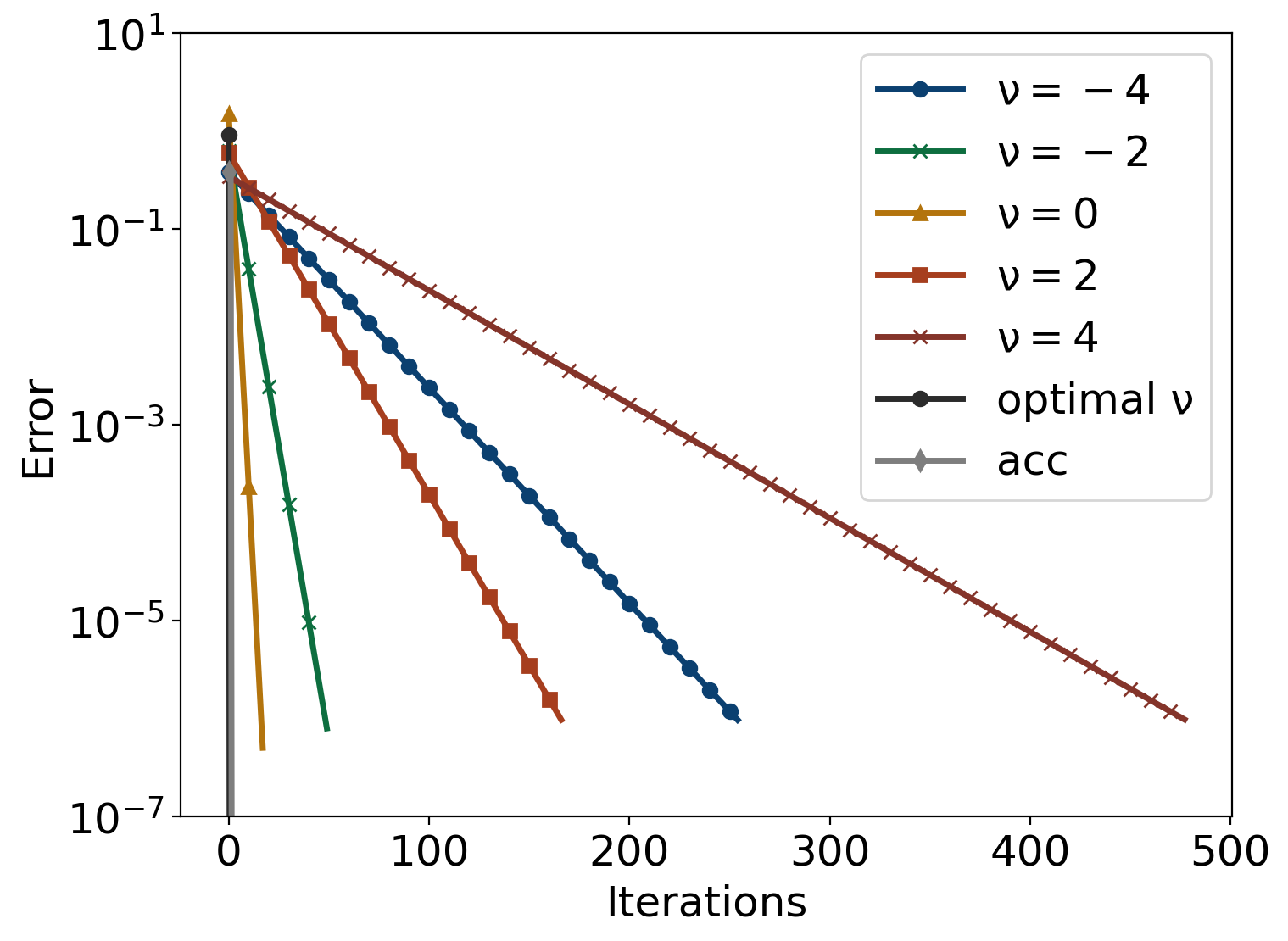}
        \caption{The difference $\|F^n-P_{\mathcal{G}\oplus \mathcal{H}}(F)\|$ after $n$ iterations}
        \label{fig:1_ChoosingModel}
    \end{subfigure}
    \hfill
    \begin{subfigure}[b]{0.45\textwidth}
        \centering
        \includegraphics[width = \textwidth]{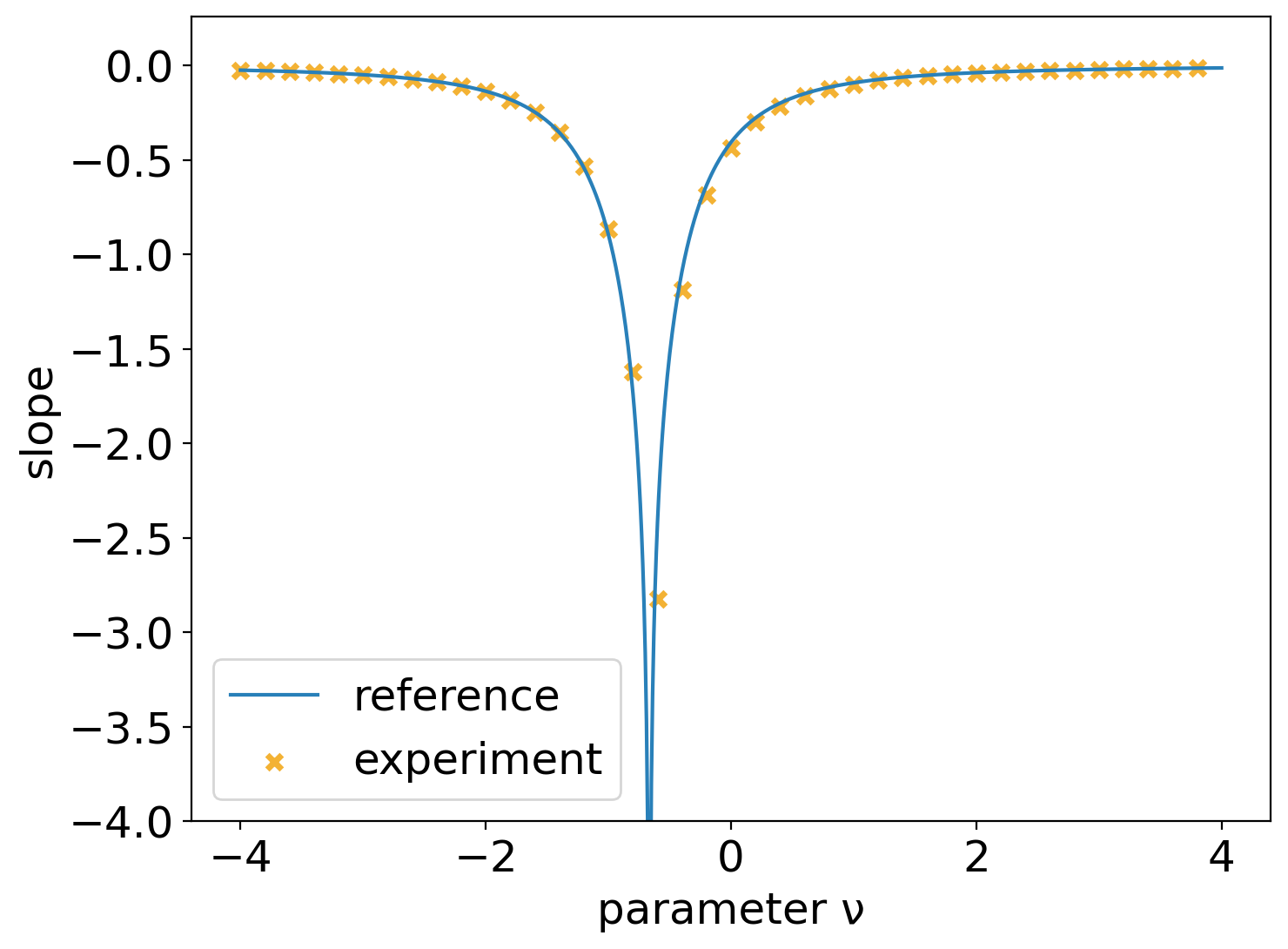}
        \caption{The convergence rate: slope of $\log(\|F^n-P_{\mathcal{G}\oplus \mathcal{H}}(F)\|)$}
        \label{fig:1_ChoosingModel_quant}
    \end{subfigure}
    \caption{Convergence rates for different hyper-parameter values ($\nu$): Left - Iteration differences by original scheme (\cref{alg:framework}) for $\nu=-4, -2, 0, 2, 4$ and optimal values, and that by acceleration scheme (\cref{alg:acc_framework});
    Right - Reference and experimental convergence rates.}
\end{figure}

\subsection{Acceleration scheme}
In \cref{sec:converge}, we have demonstrated that our algorithm is equivalent to an alternative approach involving the projection of the residual $r^n = F-F^n$ onto the complementary spaces $\mathcal{G}^{\bot}$ and $\mathcal{H}^{\bot}$ of $\mathcal{G}$ and $\mathcal{H}$, with a proof in \ref{appendix1}.
Consequently, the residual $r^n$ converges to the projection of $F$ onto the intersection of the complementary spaces $\mathcal{G}^{\bot} \cap \mathcal{H}^{\bot}$.
In a specific case where $F\in \mathcal{G}\oplus\mathcal{H}$, it yields  $P_{\mathcal{G}^{\bot} \cap \mathcal{H}^{\bot}}(F) = 0$.
To accelerate the convergence of $r^n$ towards zero, we introduce an adaptive parameter $t_F$ in front of $(r^n - r^{n-1})$, rather than employing a straightforward update like $r^n = r^{n-1} + (r^n - r^{n-1})$.
In other words, we solve the minimization problem
\begin{equation}
    t_F = \mathrm{argmin}_{\hat{t}_F}(\langle r^{n-1} + \hat{t}_F(r^n - r^{n-1}), r^{n-1} + \hat{t}_F(r^n - r^{n-1})\rangle).
\end{equation}
The optimal parameter $t_F$ is determined as:
\begin{equation}
    t_F = \frac{\langle r^{n-1}, r^{n-1} - r^n\rangle}{\langle r^{n-1} - r^n, r^{n-1} - r^n\rangle}.
\end{equation}
The acceleration algorithm, outlined in \cref{alg:acc_framework}, is based on the scheme proposed and analyzed in~\cite{bauschke2003accelerating},
and
we show here that
it has superior performance when at least one of
the hypothesis spaces is low-dimensional.
In fact, if either $\mathcal{G}$ or $\mathcal{H}$ has only one degree of
freedom, we could claim that $r^n$ and $r^{n-1}$ are guaranteed to be linearly dependent (see a detailed proof in \ref{appendix_acc}).
Hence,  
convergence is achieved with only a single step of acceleration.
Experimental results
in \cref{fig:1_ChoosingModel} validate this finding: regardless of the parameter
$\nu$, \cref{alg:acc_framework} converges after just one step.
Furthermore, the experiment in \cref{sec:example1} shows similar results.

\begin{algorithm}[ht]
    \caption{Acceleration scheme of \hm}\label{alg:acc_framework}
    \begin{algorithmic}
    \Require Data-driven models functioning as projection operators $P_{\mathcal{G}}$ and $P_{\mathcal{H}}$
    \State $n\gets 0$
    \State $F_\mathcal{H}^n\gets P_{\mathcal{H}}(F)$  \Comment{Get the best approximation of $F$ by $P_{\mathcal{H}}$}
    \State $F_\mathcal{G}^n\gets P_{\mathcal{G}}(F-F_\mathcal{H}^n)$    \Comment{Get the best approximation of $F-F_\mathcal{H}$ by $P_{\mathcal{G}}$}
    \State $F^n \gets F^n_{\mathcal{H}} + F^n_{\mathcal{G}}$
    \While{stopping criterion is not satisfied}
    \State $n\gets n+1$
    \State $F_\mathcal{H}^{n}\gets P_{\mathcal{H}}(F-F_\mathcal{G}^{n-1})$  \Comment{Get the best approximation of $F-F_\mathcal{G}$ by $P_{\mathcal{H}}$}
    \State $F_\mathcal{G}^n\gets P_{\mathcal{G}}(F-F_\mathcal{H}^n)$    \Comment{Get the best approximation of $F-F_\mathcal{H}$ by $P_{\mathcal{G}}$}
    \State $F^n\gets F_\mathcal{H}^n + F_\mathcal{G}^n$
    \State $t_F\gets \langle F-F^{n-1}, F^{n-1} - F^n \rangle/\|F^{n-1}-F^n\|^2$    \Comment{Calculate the parameter}
    \State $F_\mathcal{G}^n \gets t_F F_\mathcal{G}^n + (1-t_F) F_\mathcal{G}^{n-1}$  \Comment{Acceleration}
    \EndWhile
    \State $\tilde{F} = F_\mathcal{H}^n + F_\mathcal{G}^n$    \Comment{Obtain the best approximation in $\mathcal{H}\oplus \mathcal{G}$}
    \end{algorithmic}
    \end{algorithm}
\section{Numerical experiments}
\label{sec:experiments}
We conduct experiments in three different problem settings, involving both
predictive and control problems within parameterized dynamical systems.
The selection of models in each experiment should be analyzed on a case-by-case basis, taking into account the specific properties of the problem. The general guideline is to choose models capable of capturing the distinct characteristics of the problem. For instance, it is important to consider whether a part of the problem is linear or nonlinear, pointwise or spatially dependent, stiff or not stiff, etc. In our experiments, we address problems that combine pointwise terms with spatially dependent terms, as well as problems that involve both controlled and uncontrolled inputs.
In the numerical results, we primarily construct linear models using linear regression
and employ Koopman-type models to handle the nonlinear aspects.  We have chosen
Koopman models as an example in our numerical results due to the
well-established data-driven methods available for Koopman models and the nature
of the Koopman operator, which enables the transformation of nonlinear dynamical
systems into a linear, infinite-dimensional representation that can be further
truncated.

\subsection{Basics on Koopman models}
We now give a brief introduction of Koopman operator-based methods
to learn dynamical systems.
The reader is referred to~\cite{brunton2021modern} for a more comprehensive review.
Consider a nonlinear dynamical system $x_{n+1} = F(x_n)$ where $x_i\in X\subset
\mathbb{R}^k$ and $F: X\to X$ is a nonlinear function.
Instead of the dynamics on state spaces,
the Koopman operator approach considers
the dynamics on the space of observables,
consisting of functions $\psi : X \to \mathbb{R}$ (or $\mathbb{C}$).
We observe that if we consider the set of all such observables,
its time evolution is
$\psi(x_n) \mapsto \psi({x_{n+1}}) = \psi(F(x_{n})) \equiv \mathcal{K}\psi (x_{n})$,
where $\mathcal{K}$ is the Koopman operator, a linear operator
defined by
\begin{equation}
    \mathcal{K} \psi = \psi \circ F.
\end{equation}
To obtain a finite-dimensional representation of the Koopman operator,
we construct a finite subspace spanned by a set of
dictionary functions $\Psi = \{\psi_1, \psi_2, \psi_3, \ldots, \psi_N\}$.
Assuming this is an approximately $\mathcal{K}$-invariant subspace,
we obtain a finite-dimensional approximation of the Koopman
operator
\begin{equation}
    \Psi (x_{n+1}) = K\Psi(x_n).
\end{equation}

Data-driven Koopman models have attracted significant attention in recent years, and several mature models for learning nonlinear dynamics have been proposed, such as Extended Dynamic Mode Decomposition (EDMD)~\cite{williams2015data} and machine-learning aided
extensions~\cite{li2017extended,NIPS2017_3a835d32, lusch2018deep}.
Moreover, extended structures for parameterized Koopman
models~\cite{williams2015data,guo2023learning} and structures for control
problems~\cite{guo2023learning,korda2018linear} have been explored.
The process of approximating the evolution function can be described as a
projection onto a function subspace spanned by the observables in the
dictionary. In our experiments, the dictionary $\Psi$ should at least have the information of the state $x_n$. In other words, we could extract the state from the dictionary functions:
\begin{equation}
    x_{n+1} = g (K\Psi(x_n)).
\end{equation}
We choose the dictionary $\Psi(x)$ as $[1, x, \Phi(x)]$, where $\Phi$ represents selected or trainable observables (e.g., polynomials, radial basis functions, or neural networks). Thus, the function $g$, which aims to extract the state, can be represented as a partitioned matrix $[0, I, 0]$.


We now discuss the numerical results.  For the first two experiments, we compare
our model combination algorithm with models relying on each individual
hypothesis spaces $\mathcal{G}$ and $\mathcal{H}$, as well as residual learning.
This comparison demonstrates that our method achieves superior accuracy through
model combination. In the last experiment, we shift our focus to the advantages
of the hybrid structure in control problem and highlight the benefits of
accuracy and computational time resulting from model combination.

\subsection{Reaction-diffusion equation}
\label{sec:example1}
\paragraph{Problem}
We first return to the problem described in~\eqref{equ:0_reaction_diffusion}.
Our objective is to predict
the trajectory accurately, with the information that there is a diffusion term
and a pointwise nonlinear reaction term, and we use different hypothesis spaces
for each part.
The initial condition $u(z_i, 0)$ is drawn from a normal distribution with mean 0
and standard deviation 1 and the boundary condition is defined as $u(0,t) =
u(1,t) = 0$. To generate data for this problem, we discretize and simulate the following equation:
\begin{equation}
\frac{u^{n+1}_j - u^n_j}{\Delta t} = \mu\frac{u^n_{j+1}-2u^n_j+u^n_{j-1}}{(\Delta z)^2} + R_\eta(u^n_j)
\end{equation}
where $z_j = j/n, 1\leq j\leq n=20, u_j^n = u(z_j, t_n)$ and $R_\eta(u) = \frac{\partial }{\partial u} \left(\eta(u-1)^2(u+1)^2\right)$. In our experiments, $\mu=1$ and $\eta = \frac{1}{4}$. The time step is $\Delta t =0.0001s$.
The initial values \(u(z_j, 0)\) are chosen to follow a uniform distribution in the interval \([-1,1]\).

\subsubsection{Koopman-based model}
\label{sec:example1_koopman}
\paragraph{Methods} According to the discussions in the second paragraph of \cref{sec:illustrate_example}, the reaction-diffusion equation comprises two components: the diffusion term and the reaction term. The diffusion term could be modelled as a linear model, but it is spatially dependent. On the other hand, the reaction term could be highly nonlinear but is pointwise. Our goal is to select suitable models to approximate these two parts and then combine them. Thus, we choose two
hypothesis spaces $\mathcal{G}$, $\mathcal{H}$: one as a
linear regression model for the diffusion term,
and the other as a Koopman operator-based model for the reaction term:
\begin{equation}
    \begin{split}
        \mathcal{G}: F_\mathcal{G}(u_{j},\mu) & = \mu\frac{u_{j+1}-2u_{j}+u_{j-1}}{(\Delta z)^2},\\
        \mathcal{H}: F_\mathcal{H}(u_{j}, K) & = g\circ K\psi(u_{j}).
    \end{split}
\end{equation}
In this context, $\Psi$ represents the dictionary of the Koopman invariant
subspace, and $g$ is the function used to extract states from the dictionary.
In this experiment, we choose the dictionary as a
polynomial dictionary $\Psi = [1, u, u^2, u^3, \dots, u^{10}]$. For the training
data, we randomly choose 500 trajectories, and simulate for 10 time steps by the
forward Euler method.

\paragraph*{Results} \Cref{fig:1_IterationError} shows that the iteration difference $\|F^n - P_{\mathcal{G}\oplus \mathcal{H}}(F)\|$ by \cref{alg:framework} decreases exponentially.
Also, the acceleration scheme (\cref{alg:acc_framework}) is more efficient in approximating the reaction-diffusion problem and reaching a solution in one iteration.
To further elucidate the separation of the diffusion term and the reaction term in our scheme, we display the approximation of the reaction term $R(u)$ at various epochs (5, 10) during the training process, as well as the converged solution.
\Cref{fig:1_ParaError} and \cref{fig:1_R} emphasize how our scheme effectively separates and captures the dynamics of these two components within the problem.
These results support our theoretical analysis on the convergence of our \hm\
and demonstrate its effectiveness in approximating the reaction-diffusion
problem.
\begin{figure}[ht]
    \centering
    \begin{subfigure}[b]{0.3\textwidth}
        \centering
        \includegraphics[width = \textwidth]{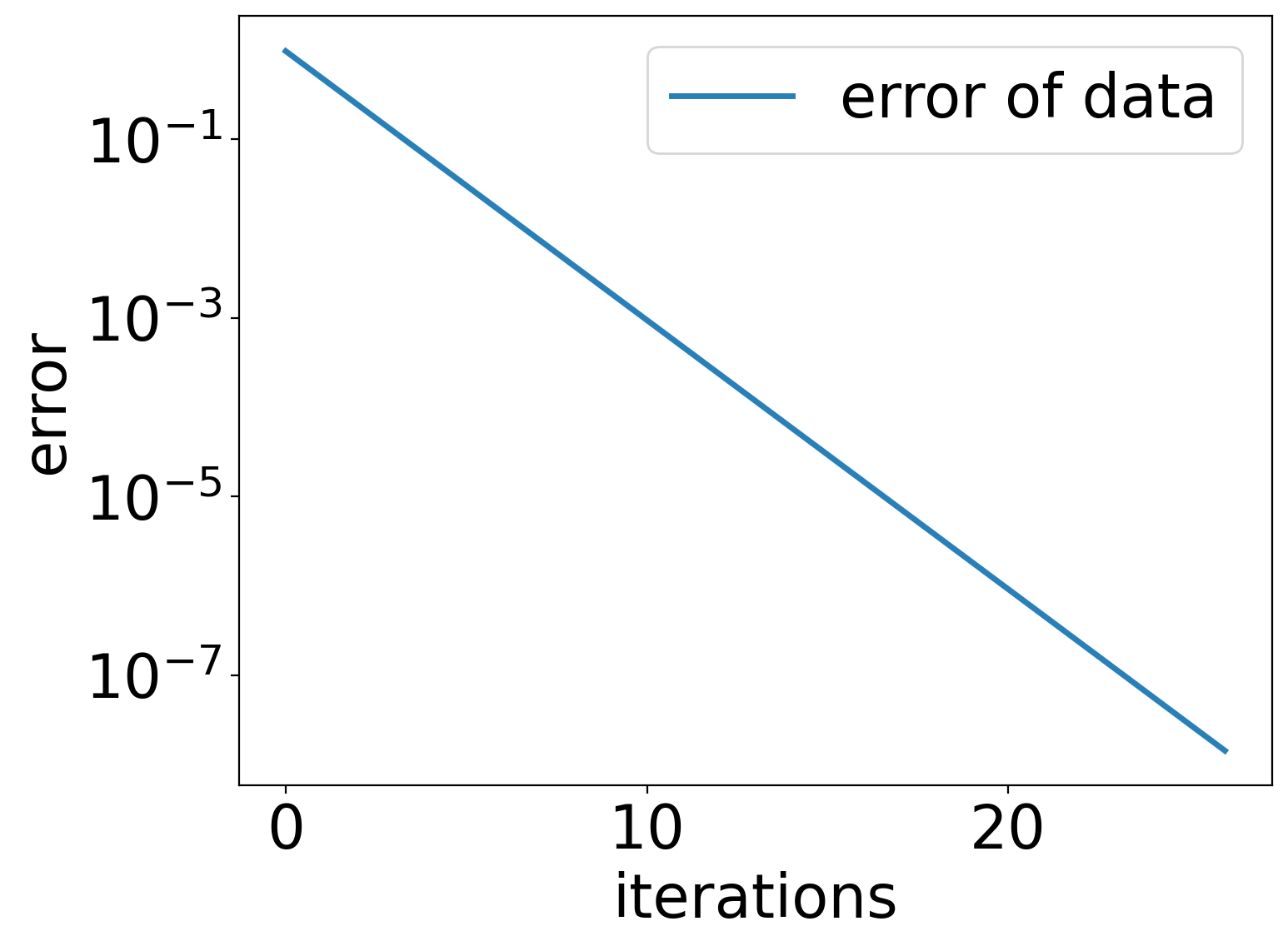}
        \caption{Iteration difference:\\ $\|F-F_\mathcal{G}-F_\mathcal{H}\|$}
        \label{fig:1_IterationError}
    \end{subfigure}
    \hfill
    \begin{subfigure}[b]{0.3\textwidth}
        \centering
        \includegraphics[width = \textwidth]{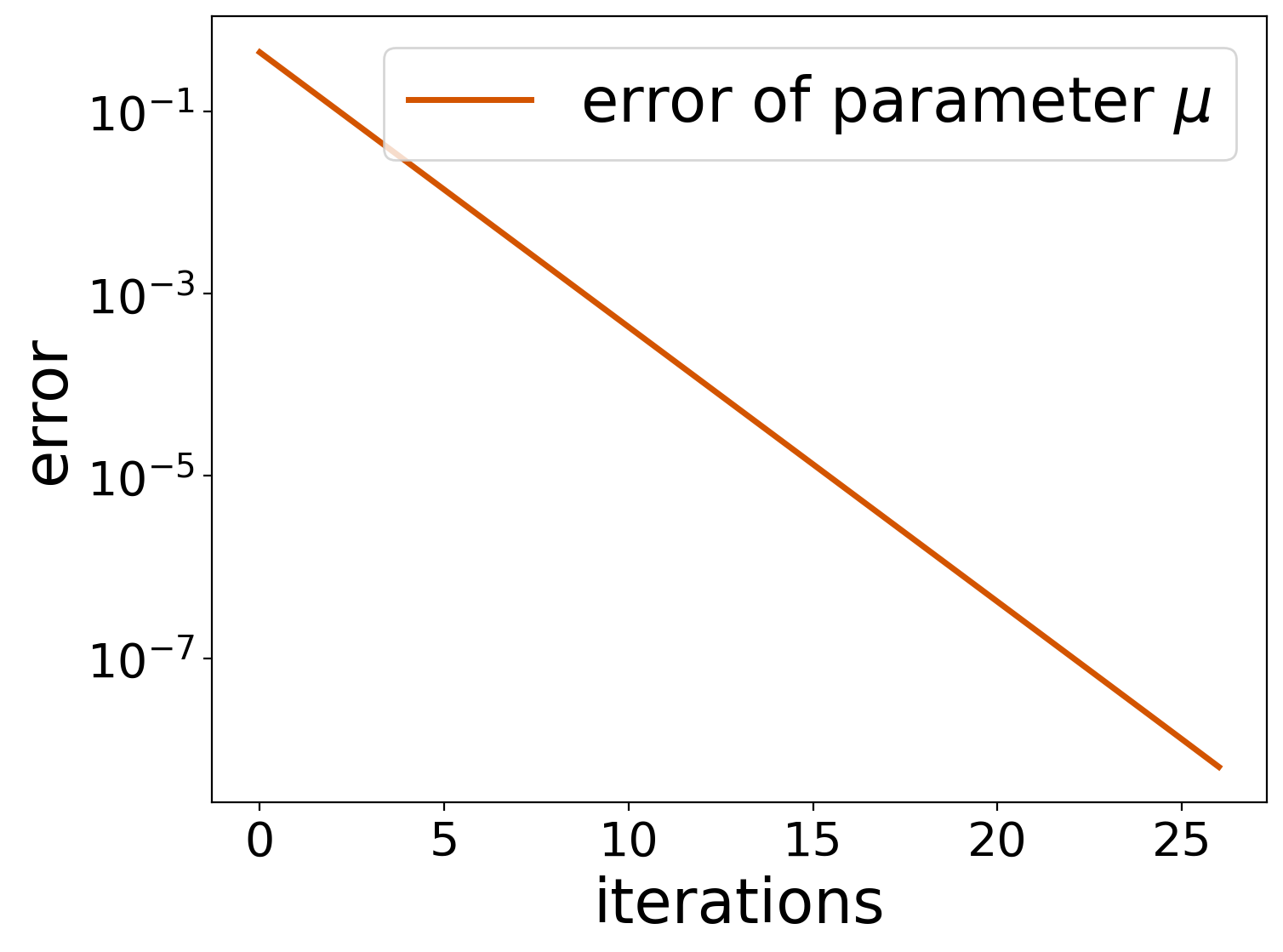}
        \caption{Error of parameter:\\ $\|\mu-\mu_{\mathrm{ref}}\|$}
        \label{fig:1_ParaError}
    \end{subfigure}
    \hfill
    \begin{subfigure}[b]{0.31\textwidth}
        \centering
        \includegraphics[width = \textwidth]{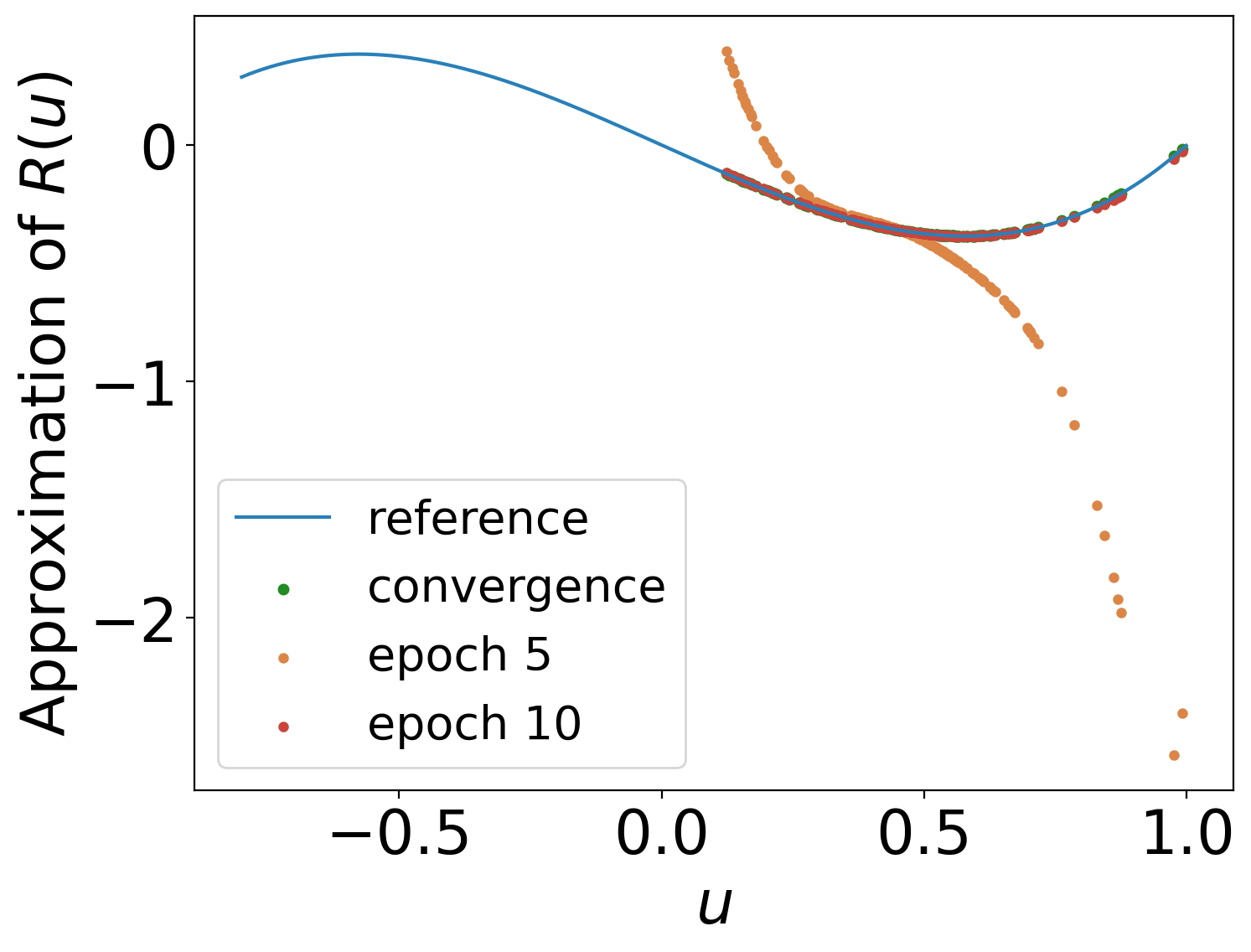}
        \caption{Convergence trajectory \\ of $R(u)$}
        \label{fig:1_R}
    \end{subfigure}
    \caption{Convergence trajectory of reaction-diffusion.}
    \label{fig:1_convergence}
\end{figure}

For the long-time prediction evaluation, we randomly select 50 initial points
$u_0$ and predict their evolution over 1000 time steps using our \hm.
We then compare the performance of our scheme with \text{three} other methods: using the Koopman operator approach with the same dictionary alone, employing a residual-learning approach that involves conducting linear regression first and then using the Koopman operator for correction, and employing gradient descent to train the two parts of the combined model together in an intrusive manner.

In \cref{fig:1_predict}, the lines represent the average error of the 50
predicted trajectories, while the shaded areas represent the standard deviation
of the estimates. Here, the error is defined as
\begin{equation}
    \mathrm{error}_k = \frac{\|u_k^{\mathrm{predict}}-u_k^{\mathrm{ref}}\|}{\|u^{\mathrm{ref}}\|},
\end{equation}
where $u_k^{\mathrm{predict}}$ represents predicted values at the $k$-th time step, and $u_k^{\mathrm{ref}}$ represents simulation values at the $k$-th time step. $u^{\mathrm{ref}}$ is constructed as $[u_0^{\mathrm{ref}}, u_1^{\mathrm{ref}}, u_2^{\mathrm{ref}}, \ldots, u_{50}^{\mathrm{ref}}]$.
 \Cref{fig:1_predict} demonstrates that combination of Koopman operator and linear regression methods can have better accuracy than baselines in prediction problems.
\begin{figure}[ht]
    \centering
    \includegraphics[width = 0.9\textwidth]{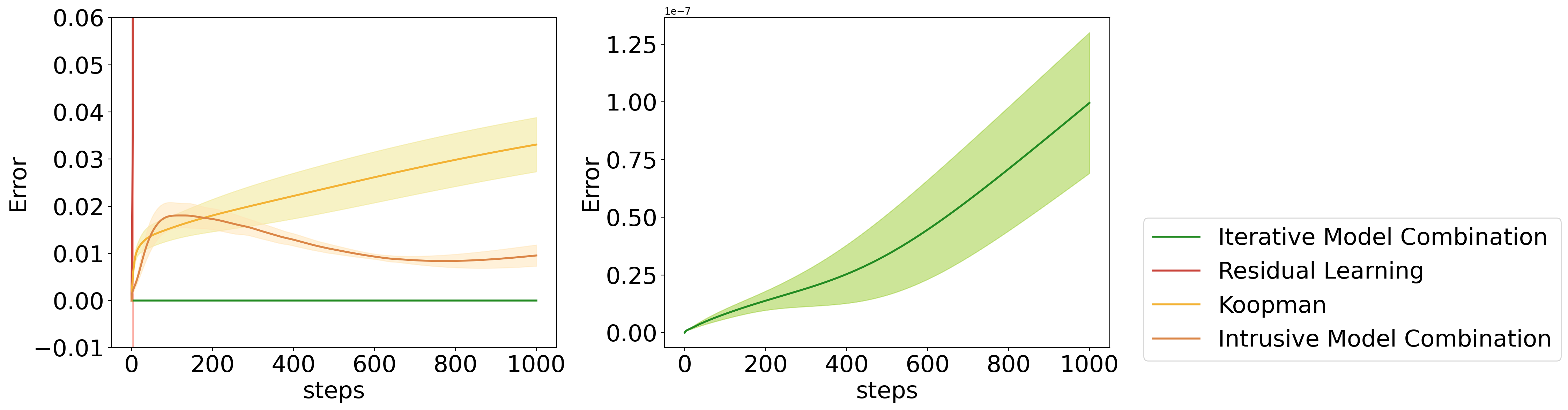}
    \caption{Long-time prediction of the reaction-diffusion solution using Koopman-based model.}
    \label{fig:1_predict}
\end{figure}

\Cref{table:1_relative_error} shows the relative error of four methods over the prediction domain,
which is the discretization of
\begin{equation}
    e_{\Omega \times [t_1, t_2]} = \frac{\sum_{(z,t)\in \Omega \times [t_1, t_2]}\|u^{\mathrm{predict}}(z,t)-u^{\mathrm{ref}}(z,t)\|}{\sum_{(z,t)\in \Omega \times [t_1, t_2]}\|u^{\mathrm{ref}}(z,t)\|}.
\label{equ:relative_error}
\end{equation}
In our experiment, $\Omega=[0,1]$, $t_1 = 0$, and $t_2 =
0.1$.
\Cref{table:1_relative_error} shows that both the Koopman model and the linear model perform poorly when used individually.
The introduction of residual learning, while resulting in some accuracy improvement at the first step, suffers from the instability in the Euler forward method, leading to rapid divergence.
In this example, both data generation and iterative model combination employ a step method equivalent to the Euler forward method. This allows our model combination to approach the best solution over numerous iterations. However, training the model directly with gradient descent achieves only a near-optimal solution. This leads to improvements in long-term predictions but remains sub-optimal.

\begin{table}[ht]
    \centering
    \begin{tabular}{ll}
    \hline \hline
        & Relative Error \\ \hline
    Linear regression      & 99.7790\%      \\
    Koopman Operator      & 74.8663\%        \\
    Residual learning          & $\geq$ 100\%         \\
    Intrusive model combination  & 38.7776\% \\
    Iterative model combination algorithm & \textbf{0.0001\%} \\
    \hline
    \end{tabular}
    \caption{Relative error over the prediction domain $[0,1]\times [0, 0.1]$.}
    \label{table:1_relative_error}
\end{table}
\subsubsection{Neural Network model}
\label{sec:example1_NN}
\paragraph*{Methods}
In the previous experiment, we modeled the reaction-diffusion equation using the Euler-forward scheme. In this experiment, we compare different numerical schemes for modeling the reaction-diffusion equation.
We choose the linear model for the diffusion term and a neural network model for the reaction term:
\begin{equation}
    \begin{split}
        \mathcal{G}: F_\mathcal{G}(u_{j},\mu) & = \mu\frac{u_{j+1}-2u_{j}+u_{j-1}}{(\Delta z)^2},\\
        \mathcal{H}: F_\mathcal{H}(u_{j}, \theta) & = NN(u_j;\theta).
    \end{split}
\end{equation}
Here, $NN(u_j;\theta)$ represents a neural network with parameters $\theta$. We construct it as a fully connected network with 2 hidden layers, each with 64 nodes. The activation function is chosen as ReLU. The input of the neural network is the state $u_j$, and the output is a rescaled approximation of the reaction term $R(u_j)$. Here, the hypothesis space $\mathcal{H}$ is no longer a linear subspace.

For different numerical schemes, the combined model is built as follows:
\begin{itemize}
    \item Forward Euler: $u^{n+1}_j = u^n_j + F_\mathcal{G}(u^n_j, \mu) + F_\mathcal{H}(u^n_j, \theta)$
    \item Backward Euler: $u^{n+1}_j = u^n_j + F_\mathcal{G}(u^{n+1}_j, \mu) + F_\mathcal{H}(u^{n+1}_j, \theta)$
    \item Central difference: $u^{n+1}_j = u^n_j + \frac{1}{2}(F_\mathcal{G}(u^n_j, \mu) + F_\mathcal{G}(u^{n+1}_j, \mu)) + \frac{1}{2}(F_\mathcal{H}(u^n_j, \theta) + F_\mathcal{G}(u^{n+1}_j, \theta))$
\end{itemize}
For the training data, we randomly select 500 trajectories and simulate each for 10 time steps using the RK45 method.
\paragraph*{Results}
In this experiment, we will obtain the optimal solution after only one iteration of training. The convergence trajectories are shown in \cref{fig:1_training_trajectory_different_scheme}. Higher-order numerical schemes, such as the central difference method, will achieve a lower iteration difference compared to other methods.
\begin{figure}[ht]
    \centering
    \includegraphics[width = 0.45\textwidth]{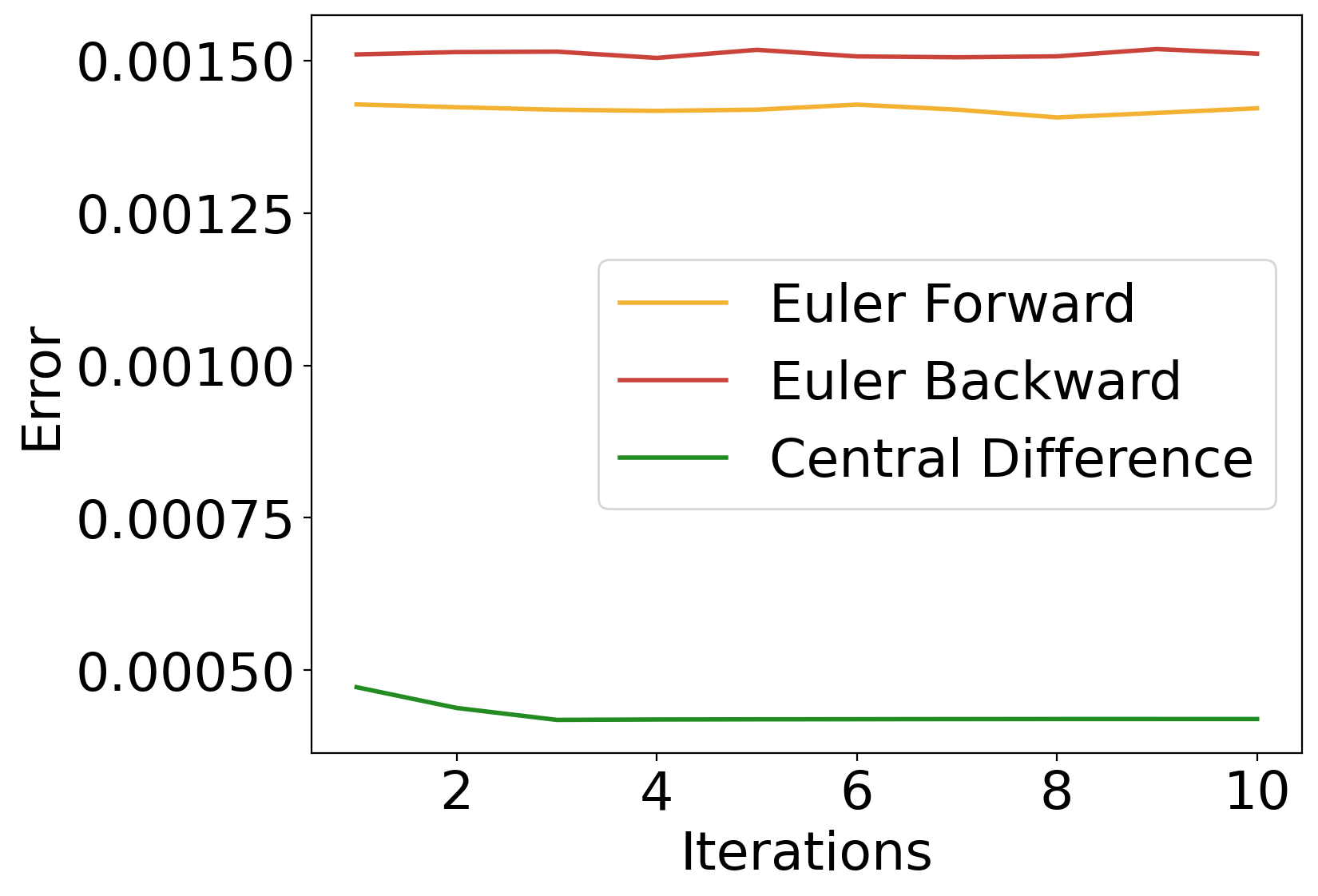}
    \caption{Iteration difference $\|F-F_\mathcal{G}-F_\mathcal{H}\|$ using different numerical schemes.}
    \label{fig:1_training_trajectory_different_scheme}
\end{figure}

For long-time prediction evaluation, we randomly select 10 initial points
$u_0$ from a uniform distribution ranging between -1 and 1
and predict their evolution over 1000 time steps using the three different numerical schemes.
\Cref{fig:1_predict_different_scheme} shows that higher-order numerical schemes (such as the central difference method) achieve better accuracy compared to the forward and backward Euler schemes. Our method is adaptable to various numerical schemes and model spaces.
\begin{figure}[ht]
    \centering
    \includegraphics[width = 0.9\textwidth]{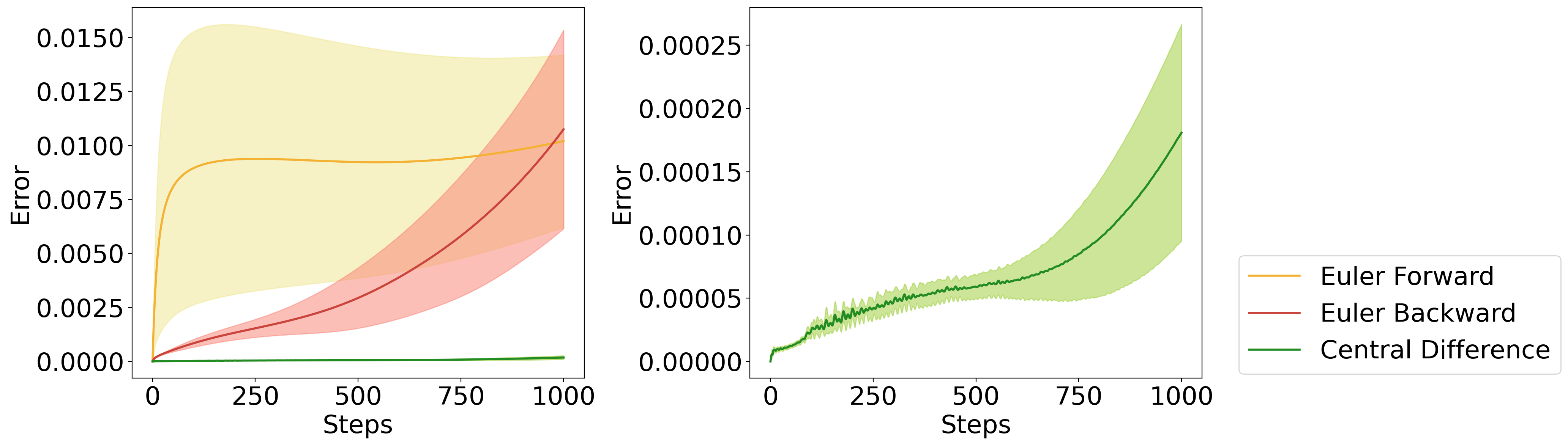}
    \caption{Long-time prediction of the reaction-diffusion solution using different numerical schemes.}
    \label{fig:1_predict_different_scheme}
\end{figure}

\subsection{Cardiac electrophysiology model}
\label{sec:example2}
\paragraph*{Problem}
Data-driven modelling of partially-known dynamical systems
is a common task in science and engineering.
One example is a PDE-ODE coupled system
in cardiac electrophysiology~\cite{sundnes2007computing,automatedPDEODE}:
\begin{equation}
    \begin{split}
        v_t-\nabla\cdot(\Lambda\nabla v)&=I_s-I_{ion}(v,s)\quad \mathrm{in}\ \Omega \times (0,T),\\
        s_t&=F(v,s)\quad \mathrm{in}\ \Omega \times(0,T).
    \end{split}
    \label{equ:ce}
\end{equation}
Here, the transmembrane potential $v$ is a one-dimensional variable, and $s$ is
a three-dimensional one, representing the gating variable of currents.
$\Lambda$ is the conductivity tensor of the mono-domain equation and $I_s$ is a
given stimulus current. The term $I_{ion}$ is the ionic current describing a
response of the cardiac cell with the gating variable $s(z,t)$ which evolves
according to~\eqref{equ:ce}. The functions $I_{ion}$ and $F$ are specific to a
given cell model as they describe the cell's dynamics. However, understanding
which properties of the cell are relevant for its electrical behavior and how
these properties interact is a challenging problem. According
to~\cite{mitusch2021hybrid}, it is thus reasonable to consider a setting where
$F$ and $I_{ion}$ would be learned from data of a specific patient. Researchers
have explored the use of neural networks to approximate the unknown components
and integrate them with the incomplete physics model in partially-known
systems~\cite{mitusch2021hybrid}. However, these methods rely on the careful
design of a physical model, particularly in~\eqref{equ:ce}, where the accurate
selection of the matrix $\Lambda$ is essential. Instead of fixing it in advance,
we choose it adaptively using \cref{alg:framework}.  Thus, in this experiment,
we suppose the matrix $\Lambda$ and the functions $I_{\mathrm{ion}}(v,s)$ and
$F(v,s)$ are unknown. The objective is to predict the trajectory of $v$ and $s$
accurately with new stimulus $I_s$.

\paragraph*{Methods}
We observe that the coupled system~\eqref{equ:ce} is composed of a reaction-diffusion equation for $v$ and an ordinary differential equation for $s$. Additionally, $v$ is influenced by the external forcing term $I_s$. Similar to the guideline of choosing methods in \cref{sec:example1_koopman}, we choose two models with hypothesis spaces $\mathcal{G}$, $\mathcal{H}$ according to properties of this specific problem: one as linear regression of fusion terms
\begin{equation}
    \begin{split}
    \mathcal{G}: F_\mathcal{G}(x(z_{i,j},t_k)) = [\mu_1 \frac{v(z_{i+1, j},t_k) + v(z_{i-1, j},t_n)-2 v(z_{i, j},t_n)}{(\Delta z)^2} \\
    + \mu_2 \frac{v(z_{i, j+1},t_k) + v(z_{i, j-1},t_k)-2 v(z_{i, j},t_k)}{(\Delta z)^2}, 0, 0, 0]^T,
    \end{split}
\end{equation}
another as Parametric Koopman Operator~\cite{guo2023learning} to tackle the reaction term with outside inputs $I_s$ for $v$ and the ODE for $s$:
\begin{equation}
    \mathcal{H}: F_\mathcal{H}(x(z_{i,j}, t_k)) = g\circ K(I_s(z_{i,j}, t_k))\Psi(y(z_{i,j}, t_k)),
\end{equation}
where $x(z_{i,j}, t_k)=[v(z_{i,j}, t_k),s(z_{i,j}, t_k)]$ and $g$ is the function used to extract states from the dictionary.

To obtain the representation of the operator $K(I_s)$, we construct a
ResNet~\cite{he2016deep} $NN_K(I_s)$, comprising two hidden layers, each with 32
nodes. The input of $NN_K (I_s)$ is the stimulus, and the output consists of the
elements of the matrix $K(I_s)$. Subsequently, we reshape the output of the
neural network to obtain the matrix $K(I_s)$.  Additionally, we utilize a
trainable dictionary~\cite{li2017extended} in our experiments. The dictionary is
constructed using a ResNet denoted as $NN_{\Psi}(x)$, which consists of 2 hidden
layers, each with 64 nodes. The input of $NN_{\Psi}(x)$ is the state of dynamics
$x$, and the output is a group of observables of $x$. We construct the
dictionary $\Psi(x)$ as
\begin{equation}
    \Psi(x) = [1, x, NN_{\Psi}(x)].
\end{equation}
In our experiment, the dictionary $\Psi(x)$ has a total dimension of 27. The structures of neural networks are illustrated in \cref{fig:2_structure}. In this experiment, we let $\Omega =[10]^2$ and $T = 5$. The domain $\Omega$ is discretized with a mesh size  $\Delta z = 0.2$, and we simulate the coupled system using a time step $\Delta t = 0.2$. For the training data, we define the inputs $I_s$ as follows:
\begin{equation}
I_s(z,t) = d(z_1, z_2)\sin(t),
\end{equation}
where $d(z_1, z_2) = z_1 + 10 - z_2$.
During the training process, we initially perform pre-training on the dictionary network $NN_{\Psi}(x)$ and freeze the parameters in $NN_{\Psi}(x)$.
Subsequently, we iteratively optimize the parameters $\mu_1$ and $\mu_2$ using linear regression and train the operator network $NN_K(I_s)$ using the Adam optimizer.
In this experiment, due to the nonlinear structure of operator network $NN_K(I_s)$, it is no longer a linear problem as before.

\begin{figure}[ht]
    \centering
    \includegraphics[width = 0.8\textwidth]{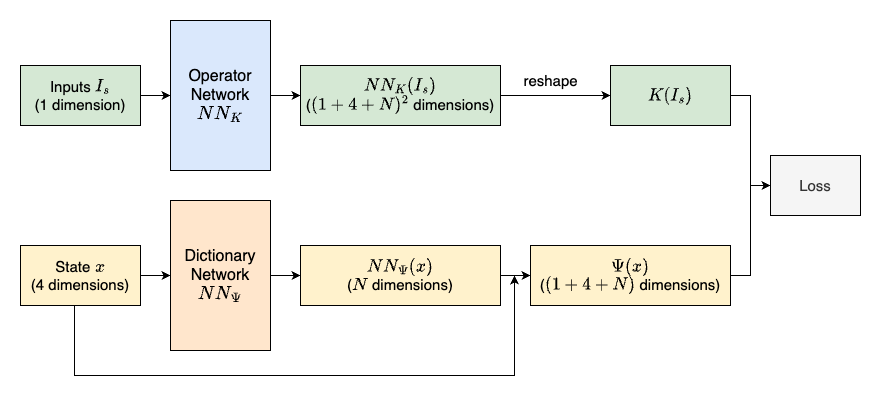}
    \caption{Structure of the parametric Koopman model.}
    \label{fig:2_structure}
\end{figure}

\paragraph{Results}
\Cref{fig:2_convergence} shows the convergence patterns of our schemes.  We
observe that initially the iteration error decreases rapidly.  However, since
the neural networks for the Koopman operator could only obtain an approximate
projection $P_\mathcal{H}(F-F_\mathcal{G})$,
the error exhibits slight fluctuations thereafter.  Also, the
acceleration scheme (\cref{alg:acc_framework}) has faster
convergence than the original scheme (\cref{alg:framework}).  Furthermore, the
relative error of parameters decreases to less than 10\% within a few epochs.
\begin{figure}[H]
    \centering
    \begin{subfigure}[b]{0.45\textwidth}
        \centering
        \includegraphics[width = \textwidth]{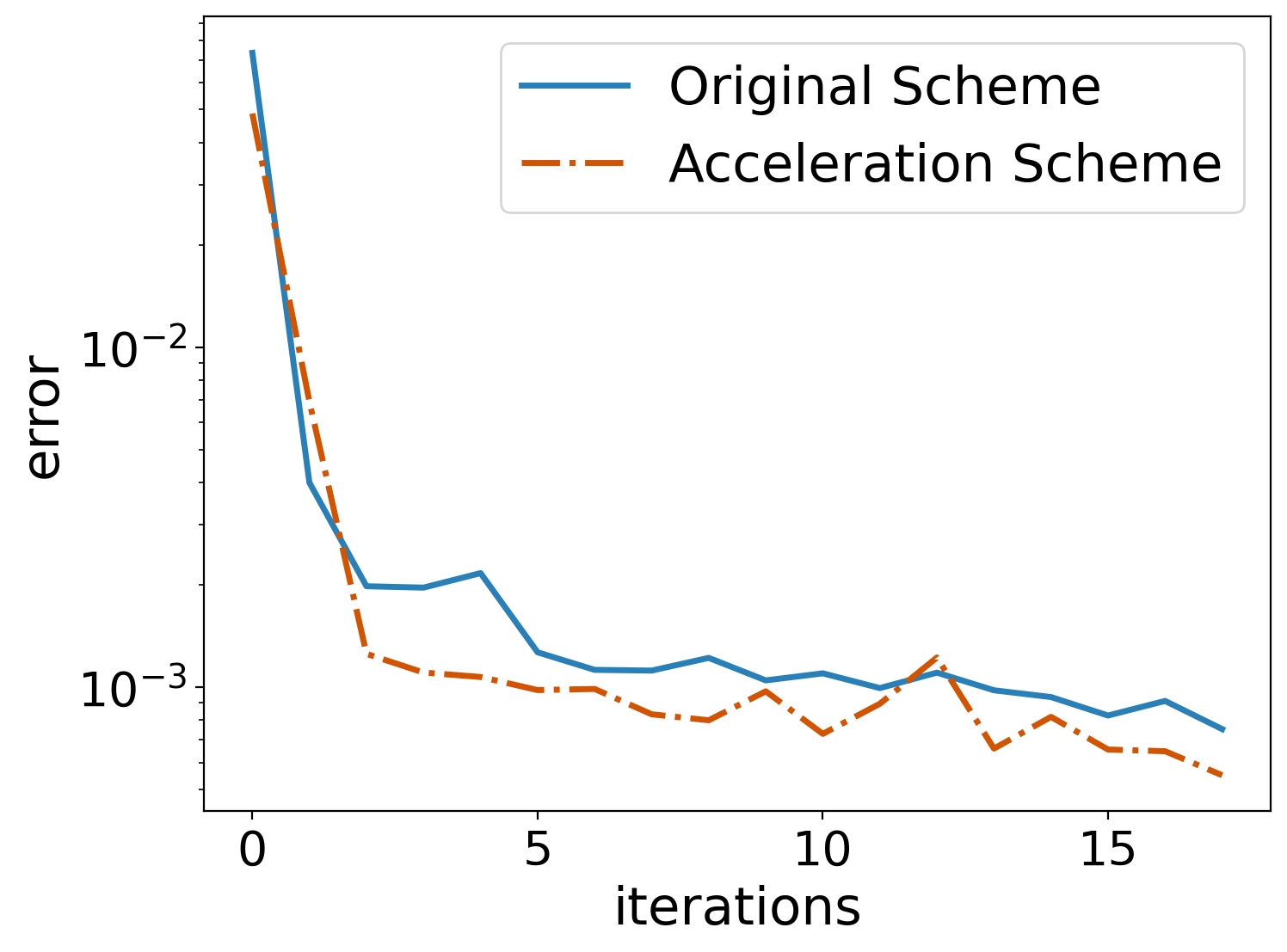}
        \caption{Iteration difference:\\ $\|F-F_\mathcal{G}-F_\mathcal{H}\|$}
        \label{fig:2_IterationError}
    \end{subfigure}
    \hfill
    \begin{subfigure}[b]{0.45\textwidth}
        \centering
        \includegraphics[width = \textwidth]{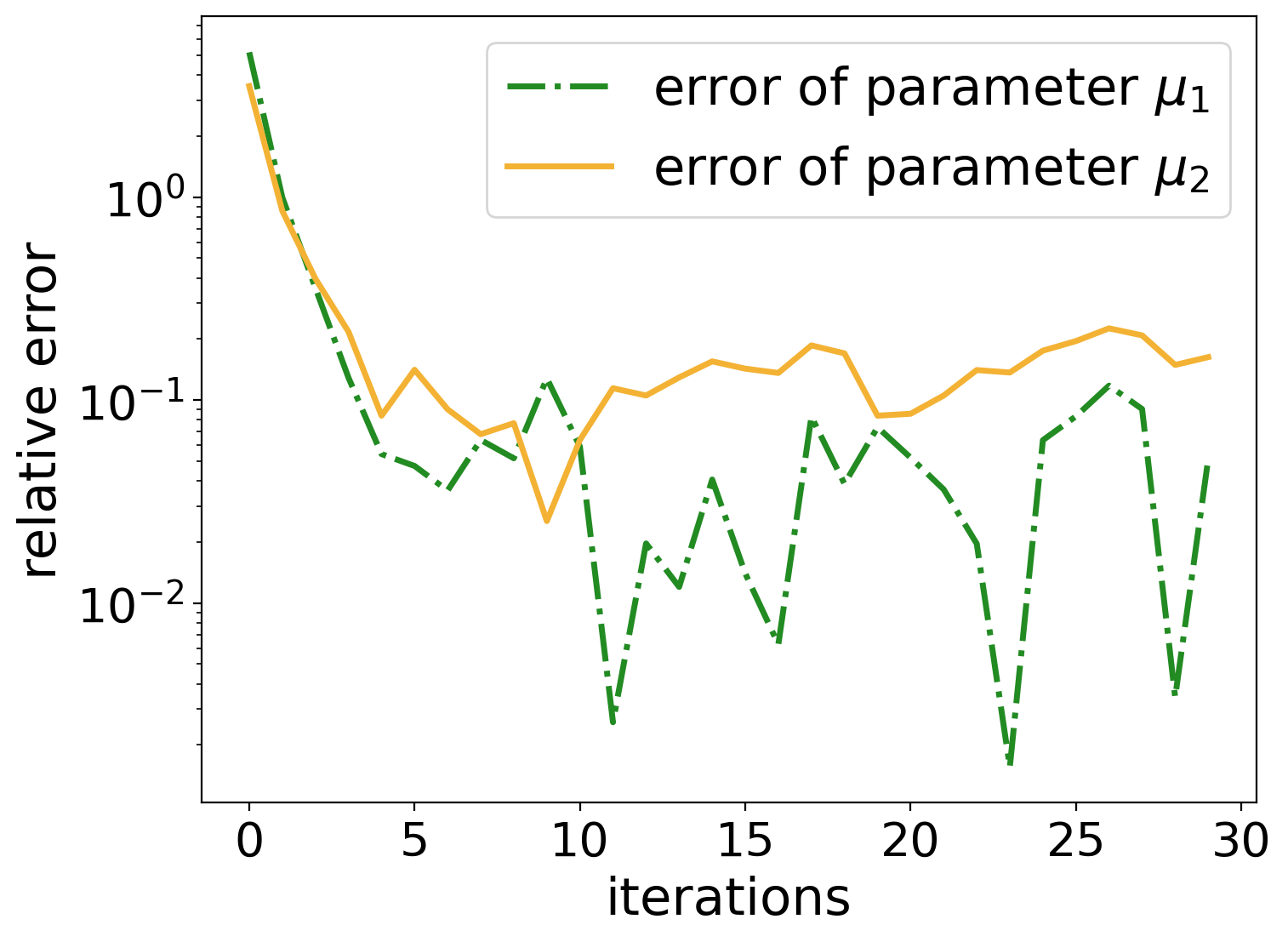}
        \caption{Error on parameter:\\ $\|\mu-\mu_{\mathrm{ref}}\|$}
        \label{fig:2_ParaError}
    \end{subfigure}
    \caption{Convergence trajectory of the cardiac model.}
    \label{fig:2_convergence}
\end{figure}

The predicted states are depicted in \cref{fig:2_train_predict}.
These predictions utilize the same initial conditions and external stimulus as those employed during training.
We chose the spatial coordinates $z = (2.8, 9.8)$ to present the predicted results because the system trajectory exhibits significant nonlinearity at this point.
We plot the predicted states at the spatial coordinates $z = (2.8, 9.8)$ in the top figures.
Both the transmembrane potential $v$ and the gating variable $s$ closely approximate the reference values.
Furthermore, the spatial distribution of the transmembrane potential $v$ at $t=3$ demonstrates the superior performance of our methods compared to using the parametric Koopman model alone.
Quantitatively, the relative error, defined in~\eqref{equ:relative_error}, in $\Omega \times [0,3]$ of the transmembrane potential $v$ and the gating variable $s$ is shown in \cref{table:2_relative_error}.
In this case, the multi-step predictive states of linear regression strictly equal the initial conditions, as the initial conditions for $v$ and $s$ over the spatial domain $\Omega$ are chosen to be constant.
Furthermore, since the iteration scheme in model combination employs the forward-Euler scheme of a dynamical system, which is not unconditionally stable,
the residual learning in this experiment is significantly affected by stiffness, causing the error to explode after a few steps.
When comparing our method to the parametric Koopman model, we can observe a significant performance improvement in our approach.
We also compare the performance of our method with the intrusive model combination approach, which involves directly training the two parts of the combined model together using the Adam optimizer.
The results demonstrate that our method achieves superior accuracy compared to the intrusive method.
This is attributed to our ability to more carefully address the rescaling problem during training by separating the two parts of the model.
This careful handling of rescaling represents a potential advantage of our approach when combining models with different scales.

\begin{figure}[ht]
    \centering
    \includegraphics[width = \textwidth]{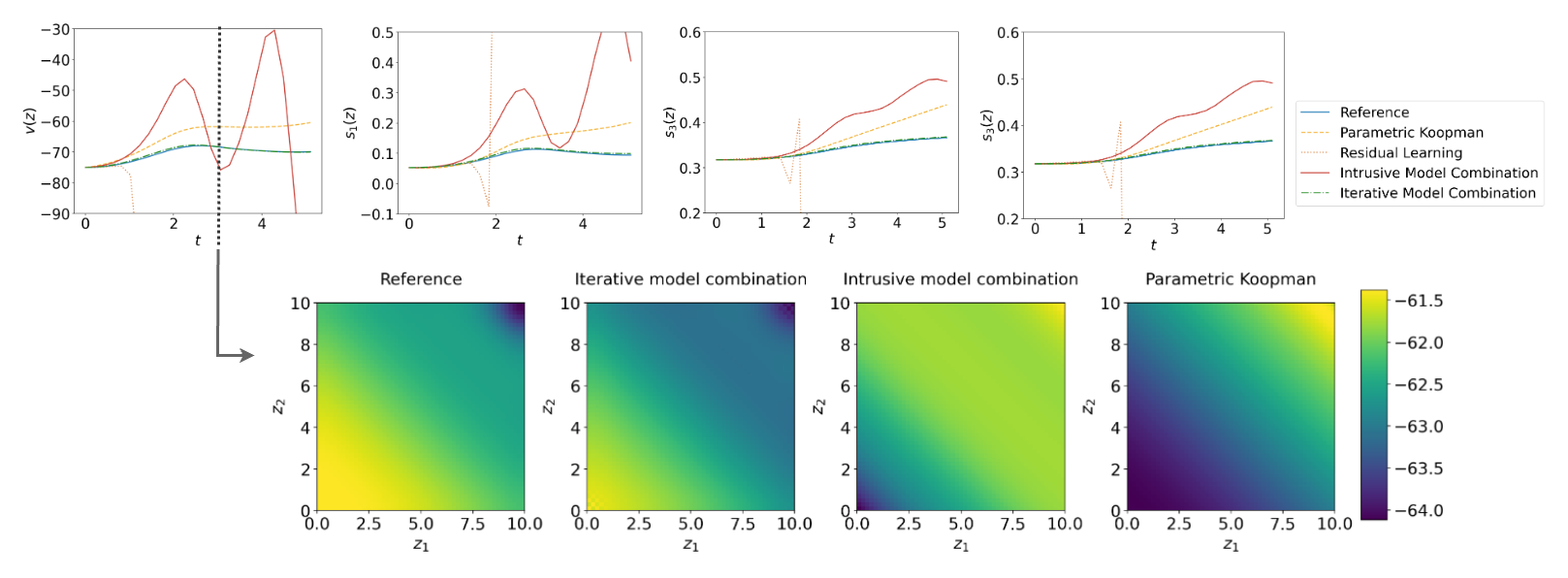}
    \caption{Predicted states with same initial conditions and external stimulus.
    Top: Predicted states at spatial coordinates $z = (2.8, 9.8)$.
    Bottom: Spatial distribution of the transmembrane potential $v$ at $t=3$. Left: the reference distribution; Middle: by \hm; Right: by Parameterized Koopman.}
    \label{fig:2_train_predict}
\end{figure}

\begin{table}[ht]
    \centering
        \begin{tabular}{lllll}
        \hline \hline
                                     & $v$                 & $s_1$                & $s_2$             & $s_3$                \\ \hline
        Linear regression      & 10.4159\%         & 58.5347\%         & 8.2891\%          & 8.2365\%          \\
        Parametric Koopman Model & 4.3523\%          & 13.4181\%         & 3.0881\%          & 1.9963\%          \\
        Residual learning            & $\geq 100\%$        & $\geq 100\%$   & $\geq 100\%$    & $\geq 100\%$  \\
        Intrusive model combination & 23.4798\% & 95.9591\% & 24.1986\% & 9.5839\% \\
        \Hm & \textbf{0.3269\%} & \textbf{1.6881\%} & \textbf{0.5201\%} & \textbf{0.2469\%} \\ \hline
        \end{tabular}
    \caption{Relative error of prediction in $\Omega \times [0,3]$ with same initial conditions and external stimulus.}
    \label{table:2_relative_error}
    \end{table}

    We also tested the models with a new stimulus, denoted as $I_s(z,t) = d(z_1, z_2) \cos(t)$.
    \Cref{fig:2_test_predict} displays the predicted state at the spatial coordinates $z = (2.8, 9.8)$ and predicted spatial distribution at $t=2$,
    illustrating our method's capability to handle previously unseen inputs.
    Quantitatively, \cref{table:2_relative_error_test} presents the relative errors, defined in~\eqref{equ:relative_error},
    in $\Omega \times [0,2]$ for both the transmembrane potential $v$ and the gating variable $s$ with different external stimuli.
    We can observe that the results are similar to the predictive experiment with the same external stimulus, emphasizing that our method can also exhibit better performance when faced with unseen stimuli.

    \begin{figure}[ht]
        \centering
        \includegraphics[width = \textwidth]{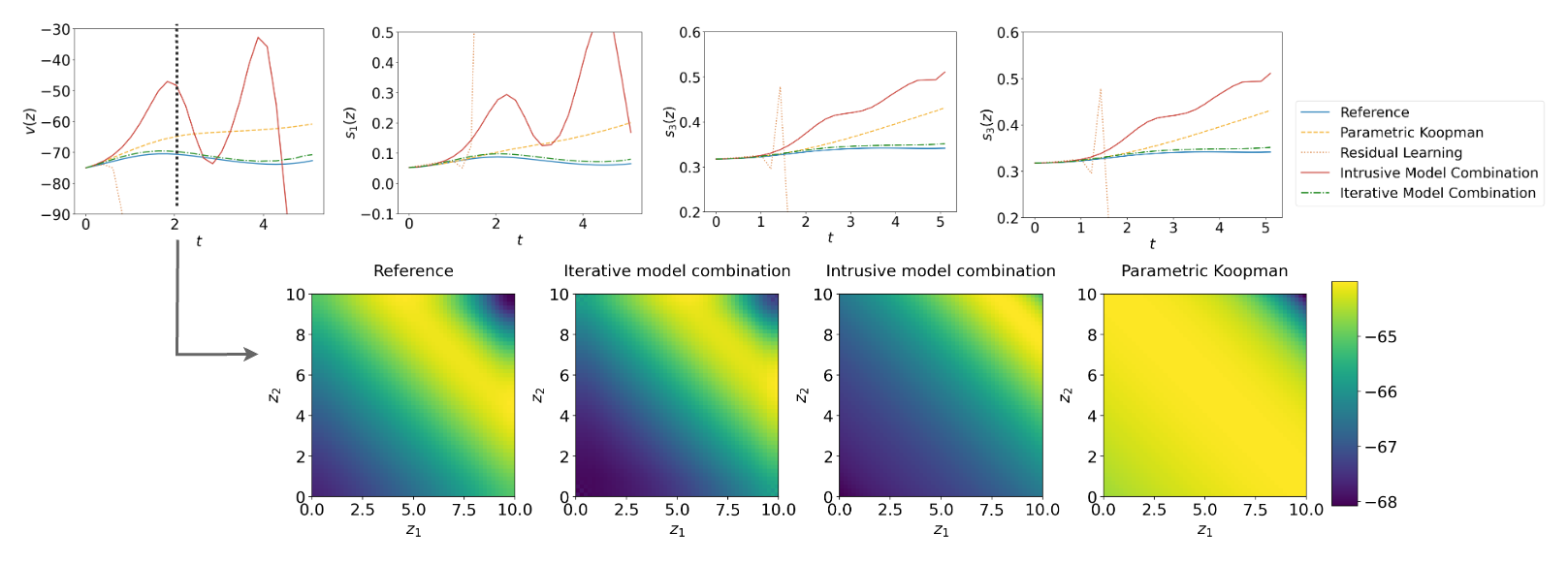}
        \caption{Predicted states with different external stimuli.
        Top: Predicted states at spatial coordinates $z = (2.8, 9.8)$.
        Bottom: Spatial distribution of the transmembrane potential $v$ at $t=2$. Left: the reference distribution; Middle: by hybrid-modelling; Right: by Parameterized Koopman.}
        \label{fig:2_test_predict}
    \end{figure}

\begin{table}[ht]
\centering
\begin{tabular}{lllll}
\hline\hline
                                    &$v$     & $s_1$                & $s_2$                & $s_3$                \\ \hline
            Linear regression          & 8.9919\%          & 5.7578\%         & 7.1977\%          & 6.0076\%          \\
            Parametric Koopman Model   & 5.0429\%          & 1.4269\%         & 4.9132\%          & 1.3653\%          \\
            Residual learning          & $\geq 100\%$        & $\geq 100\%$        & $\geq 100\%$        & $\geq 100\%$        \\
            Intrusive model combination & 17.7203\% & $\geq 100\%$ & 18.0863\% & 9.3989\% \\
            \Hm & \textbf{0.9558\%} & \textbf{7.8112\%} & \textbf{0.3870\%} & \textbf{0.8849\%} \\ \hline
            \end{tabular}
            \caption{Relative error of prediction in $\Omega \times [0,2]$ with different external stimuli.}
    \label{table:2_relative_error_test}
            \end{table}


\subsection{Model predictive control in parameterized systems}
\label{sec:example3}

\paragraph*{Problem}
In this section, we consider a parameterized discrete-time nonlinear controlled
dynamical system
\begin{equation}
    x_{n+1} = F(x_n, c_n, e_n),
\end{equation}
where \(x_n\in X\subset\mathbb{R}^K\) denotes the state of the system, \(c_n\in
\mathcal{C}\subset \mathbb{R}^{m_c}\) is the control input, \(e_n\in
\mathcal{E}\subset \mathbb{R}^{m_e}\) represents the external input, and \(F\) is the transition mapping. The
control input \(c_n\) remains subject to adjustments and monitoring by the
controller, whereas \(e_n\) is influenced by external factors and
cannot be controlled.

An example for such control problem is the dynamics of a
robotic fish studied in~\cite{mamakoukas2019local}.
The states of the robotic fish are $s = [x, y, \phi, v_x, v_y, \omega]^T$, where $x, y$ are the world-frame coordinates, $\phi$ is the orientation, $v_x$ and $v_y$ are the body-frame linear velocities (surge and sway, respectively), and $\omega$ is the body-frame angular velocity.
We use $\alpha$ to indicate the angle of the tail. The tail is actuated with $\alpha(t) = \alpha_0 + \alpha_a \sin(\omega_a t)$, where $\alpha_a, \alpha_0, \omega_a$ are the amplitude, bias, and frequency of the tail beat.
We use the bias $\alpha_0$ and amplitude $\alpha_a$ as control inputs, while the frequency $\omega_a$, which may be determined by battery and mechanical constraints, will be used as an external input in the simulation.
The dynamical system is defined as an average model for tail-accurate systems~\cite{wang2015averaging}:
\begin{equation}
\label{equ:robotics}
\dot{s}=\left[\begin{array}{c}
\dot{x} \\
\dot{z} \\
\dot{\psi} \\
\dot{v}_x \\
\dot{v}_y \\
\dot{\omega}
\end{array}\right]=F(s) \triangleq\left[\begin{array}{c}
v_x \cos (\psi)-v_y \sin (\psi) \\
v_x \sin (\psi)+v_y \cos (\psi) \\
\omega \\
f_1(s)+K_f f_4\left(\alpha_0, \alpha_a, \omega_a\right) \\
f_2(s)+K_f f_5\left(\alpha_0, \alpha_a, \omega_a\right) \\
f_3(s)+K_m f_6\left(\alpha_0, \alpha_a, \omega_a\right)
\end{array}\right]
\end{equation}
with detailed information of parameters and functions in~\cite{mamakoukas2019local}.

We focus on the tracking problem with~\eqref{equ:robotics}.
The problem is formulated as follows: given the initial value $s_0$ and external inputs $\{w_i\}$,
we would like to obtain a sequence of controlled inputs $\{c_i\}$ minimizing
$\sum_{i=1}^n ((s_i-s^{\mathrm{ref}}_i)^T Q (s_i-s^{\mathrm{ref}}_i)+ c_i^T R c_i)$
where $s_{i+1} = f(s_i, c_i, e_i)$ and $Q, R$ are semi-positive matrices.
We solve this tracking problem by Model Predictive Control (MPC) where the current control is obtained by solving a finite horizon open-loop optimal control problem at each sampling instant.
The MPC formulation solves the following optimization problem:
\begin{equation}
    \begin{aligned}
        \min_{c_0, c_1, \ldots, c_{h-1}} J(z, c, e)\\
        \text{s.t.} \quad z_i = \hat{f}(z_{i-1}, c_{i-1}, e_{i-1})\\
        c_{i-1} \in [a_{i-1}, b_{i-1}]
    \end{aligned}
    \label{equ:4_mpc}
\end{equation}
where $J(z,c) = \sum_{i=1}^h(z_i-z^{ref}_i)^T Q (z_i-z^{ref}_i) + \sum_{i=1}^h{c_i}^T R c_i$. The MPC solver is outlined in \cref{alg:mpc}.
For more information on the MPC method, the reader may consult
the review~\cite{mayne2000constrained}.
\begin{algorithm}
    \caption{MPC solver in parameterized dynamical system}
    \label{alg:mpc}
    \begin{algorithmic}
        \Require Initial value \(x_0\), dictionary \(\Psi\), horizon length \(h\), trajectory length \(N\), reference trajectory \(\mathbf{x}^{\mathrm{ref}}\), external inputs \(\mathbf{\hat{c}}\)
        \State \(z_0\gets \Psi(x_0)\)
        \State $\mathbf{z}^{\mathrm{ref}} = \Psi(\mathbf{x}^{\mathrm{ref}})$
        \For{\(i = 1\) to \(N\)}
            \State Solve the open-loop optimal control problem from \(i\) to \(i+h\) steps
            \State \(c_{i} \gets \text{OptimalControlSolver}(z_{i-1}, i, h, \mathbf{z}^{\mathrm{ref}}[i,i+h])\) \Comment{Solve optimization problem}
            \State Propagate state: \(x_i \gets F(x_{i-1}, c_{i}, e_{i})\) \Comment{Update state using dynamics}
            \State \(z_i\gets \Psi(x_i)\)
        \EndFor
    \end{algorithmic}
\end{algorithm}

The fundamental challenge in an MPC solver lies in selecting the appropriate transition function of observables, denoted as \(z_{i+1} = \hat{f}(z_i)\) in \cref{alg:mpc}.
Unlike the previous two experiments, in this section we compare the models with different structures for solving the control problem, highlighting the fact that our method of model combination can provide more accurate solutions when dealing with the control problem and guarantee the convexity of the optimization problem.

\paragraph*{Methods} Some data-driven methods are proposed for the linearization
of the MPC problem. Two representative structures are a linear
structure~\cite{korda2018linear} as $z_{n+1} = Kz_n + Bc_n$ and nonlinear
structure~\cite{williams2015data, guo2023learning} as $z_{n+1} = K(c_n)z_n$. Given that the
control is on \(c\), our model combination strategy
can ensure that the dynamics is linear in \(c\)
while a nonlinear structure is used to handle the external
input \(e\). We express this as follows:
\begin{equation}
    z_{n+1} = K(e_{n}) z_n + B c_n.
\end{equation}
Another approach to merging these two structures is represented by:
\begin{equation}
    z_{n+1} = Kz_n + B(e_{n}) c_{n},
\end{equation}
as $c_n$ could also be treated as part of the dictionary.

Thus, we compare the performance of the following four structures,
where only the hybrid ones are trained using the proposed method:
\begin{itemize}
    \item Linear structure, similar to~\cite{korda2018linear}: \(z_{n+1} = K_1z_n + B_1c_n + C_1e_n\),
    \item Hybrid structure (our methods, represented as hybrid structure 1 and hybrid structure 2 respectively):
    \begin{itemize}
        \item \(z_{n+1} = K_2(e_n)z_n + B_2c_n\),
        \item \(z_{n+1} = K_3z_n + B_3(e_{n}) c_{n}\),
    \end{itemize}
    \item Nonlinear structure, similar to~\cite{williams2015data, guo2023learning}: \(z_{n+1} = K_4(e_n,c_n)z_n\).
\end{itemize}
The optimization problem in~\eqref{equ:4_mpc} with linear and hybrid structures
is convex, with a guarantee in obtaining an optimal solution, whereas with
nonlinear structure it is non-convex. We parametrize the observables using a
fully-connected neural network denoted as \(NN_{\Psi}(s)\), and establish the
dictionary as \(\Psi(s) = [1, s, NN_{\Psi}(s)]\).
The dimension of the
dictionary \(\Psi(s)\) is randomly picked as 25 in this experiment.

Throughout our experiments, we follow a two-step process: initially, we perform
pre-training to learn the dictionary and subsequently fix it; then, we proceed
to train the operators.  The traditional Koopman operator (\(K_1, B_1, C_1, B_2,
K_3\)) is computed via least squares, while the parameterized Koopman operator
(\(K_2(e), B_3(e), K_4(c,e)\)) is constructed using a fully connected neural
network.  In the case of the hybrid structure, we use the Adam optimizer to
train the parameterized operator network for 5 epochs and perform a least
squares update on the traditional Koopman operator by \cref{alg:framework}. This
process is repeated for a total of 200 iterations. On the other hand, for the
nonlinear structure, we train the operator network for 1000 epochs.

To generate data, we sample $P = 1000$ initial conditions for the states with uniform distributions given by \cref{table:4_initial}.
For each individual sample, we apply random inputs generated from uniform distributions: $U_{\alpha_0}\left(-50^{\circ}, 50^{\circ}\right)$ for the tail angle bias, $U_{\alpha_a}\left(0,30^{\circ}\right)$ for the tail angle amplitude of oscillations, and $U_{\omega_a}\left(0,10\pi\ \mathrm{rad/s}\right)$ for the frequency.
Then, for each specific combination of initial conditions $s_k$, controlled inputs $u_k$, and external inputs $w_k$, we employ the dynamics described in~\eqref{equ:robotics}, we generate data for a total of 50 discrete time steps with a time step $\Delta t=1\mathrm{s}$.
Additionally, for the tracking problem, we construct the reference trajectory using constant inputs: $\alpha_0 = \frac{1}{6}\pi$, $\alpha_a = \frac{1}{18}\pi$, and $\omega_a = 2\pi$.
Subsequently, we leverage the linear, hybrid, and nonlinear structures to track the trajectory under distinct external input conditions for $\omega_a$, which are set to $2\pi$, $2\pi + \pi \sin(\frac{\pi}{60}t)$, and random samples uniformly selected in $[0, 10\pi]$, respectively.
In the MPC problem, we choose the horizon length and cost function as follows: $h = 5$, $Q = \mathrm{diag}\{0,0,0,0,10^8,0.01,10^8,0,\ldots,0,\ldots\}$, $R = \mathrm{diag}\{0.1,0.1\}$ and the total length of trajectory is 180. The optimization problems are solved by the L-BFGS-B method.

\begin{table}[ht]
    \centering
    \begin{tabular}{ll}
    \hline
    \hline
   Parameter & Distribution \\ \hline
    $x$  & $\mathcal{U}(-0.5, 0.5)$\\
    $y$          &$\mathcal{U}(-0.1, 0.1)$\\
    $\phi$          &$\mathcal{U}(-\pi/4, \pi/4)$\\
    $v_x$          &$\mathcal{U}(0, 0.04)$\\
   $v_y$          & $\mathcal{U}(-0.0025, 0.0025)$ \\
   $\omega$          &$\mathcal{U}(-0.5, 0.5)$\\ \hline
\end{tabular}
\caption{Distribution of initial values.}
\label{table:4_initial}
\end{table}

\paragraph{Results}
The results of the tracking experiment are visually depicted in
\cref{fig:4_track}.
In the tracking problem described in the previous paragraph,
we are primarily concerned with only two dimensions, namely $v_x$ and $\omega$ (the 5th and 7th dimensions of $\Psi(s)$),
as they play a crucial role in maintaining the shape of the robotic fish's trajectory.
The figure on the left displays the mean tracking error of $v_x$ and
$\omega$, which is defined as:

\begin{equation}
\text{mean tracking error} = \frac{1}{2} \left( \frac{\max_i|v_x(t_i) - v_x^{\text{ref}}(t_i)|}{\max_i|v_x^{\text{ref}}(t_i)|} + \frac{\max_i|\omega(t_i) - \omega^{\text{ref}}(t_i)|}{\max_i|\omega^{\text{ref}}(t_i)|} \right).
\end{equation}
Additionally, the computation time of the MPC problem is shown on the right. The linear structure has the benefit of simplicity and the fastest solving time. However, when the external inputs $\omega_a$ change rapidly (e.g., when $\omega_a$ is chosen randomly), the predicted trajectory exhibits significant oscillations. The nonlinear structure is sometimes susceptible to optimization errors and exhibits poor accuracy even in simple problem settings. In comparison to these two, our methods demonstrate good performance with a modest time cost.

\begin{figure}[ht]
    \centering
    \includegraphics[width = 0.95\textwidth]{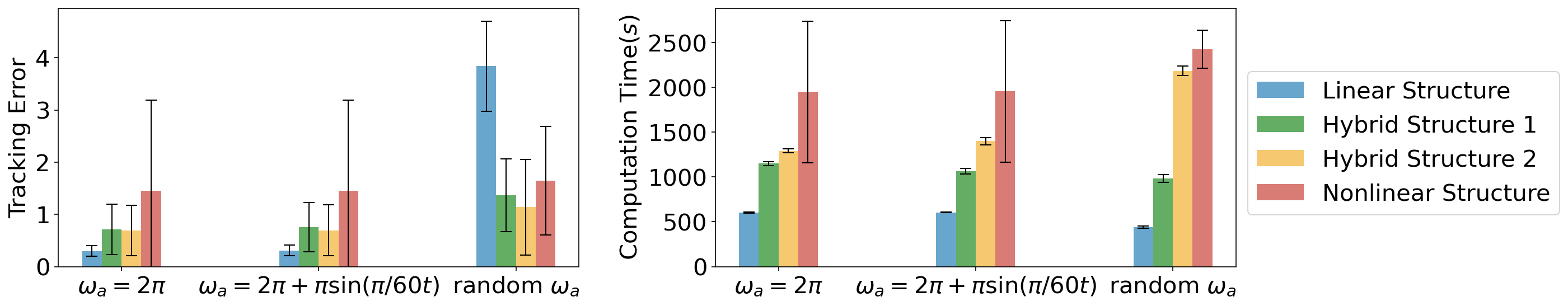}
    \caption{Tracking problem for different external inputs $\omega_a$. Left: mean tracking error of $v_x$ and $\omega$. Right: computation time of the MPC problem.}
    \label{fig:4_track}
\end{figure}

\subsection{Discussion}
In this section, we discuss the advantages and limitations of the proposed methodology as presented in the experiments.
The first advantage is that our method can identify the optimal solution in the direct-sum space $\mathcal{G} \oplus \mathcal{H}$ of hypothesis spaces $\mathcal{G}$ and $\mathcal{H}$. 
When $\mathcal{G}$ and $\mathcal{H}$ satisfy \cref{ass:1} and both training methods for finding the optimal solution in the individual hypothesis spaces $\mathcal{G}$ and $\mathcal{H}$ are able to identify the precise optimal solution, then the optimal solution obtained by our method and its acceleration scheme is also precise. 
This is demonstrated in the experiment described in \cref{sec:example1_koopman}. 
Additionally, even when the hypothesis spaces $\mathcal{G}$ and $\mathcal{H}$ are not linear subspaces, the results presented in \cref{sec:example1_NN} and \cref{sec:example2} demonstrate that the proposed method remains effective.

Beyond optimal solution identification, results presented in \cref{sec:example2} demonstrate that our method offers additional benefits when combined with neural network-based methods, as compared to intrusive methods.
That is because our method enables careful handling of scaling in the training of neural networks.
For illustration, let us assume that $\mathcal{H}$ is a neural network hypothesis.
When learning $F = F_\mathcal{G} + F_\mathcal{H}$, if $F_\mathcal{G}(x)$ present a much larger scale then $F_\mathcal{H}(x)$ in the training dataset, $\hat{F}_\mathcal{H} = 0$ could be a plausible solution when training together.
However, our method could present scaling information, allowing for rescaling the residual $F-\hat{F}_\mathcal{G}$ each time $\hat{F}_\mathcal{H}$ is trained.
This rescaling process can significantly enhance the performance of the optimization process.

The main limitation of our method is that to ensure the non-intrusiveness of the training process, the two components of the combined model must be separable. 
For instance, in a multi-step simulation process such as $x_1 = F_\mathcal{G}(x_0) + F_\mathcal{H}(x_0)$ followed by $x_2 = F_\mathcal{G}(x_1) + F_\mathcal{H}(x_1)$, when we optimize $F_\mathcal{G}$ and $F_\mathcal{H}$ separately using the data tuples $\{x_0, x_2\}$, the two components of the combined model are not separable.
Consequently, the non-intrusiveness of the training process cannot be guaranteed under these circumstances.
Future work could aim to extend the method, both theoretically and algorithmically, to address problems involving non-separable models.
This might involve approaches related to operator splitting~\cite{macnamara2016operator}.

\section{Conclusion}
\label{sec:conclusion}
Model combination holds significant potential in modelling
high-dimensional and complex dynamical systems. In this paper, we introduced
a linearly convergent \hm, an approach that efficiently combines two
model spaces through addition in a non-intrusive way. We provided both
theoretical and empirical evidence to support its convergence patterns.
Furthermore, our numerical experiments confirm the effectiveness of our approach in
addressing both prediction and control problems within complex systems. In future
studies, it would be interesting to explore extending our method to incorporate
unsupervised learning for model selection and acceleration strategies, and to apply the methodology on engineering cases of interest.

\section*{Acknowledgments}
This research is part of the programme DesCartes and is supported by the National Research Foundation, Prime Minister’s Office, Singapore under its Campus for Research Excellence and Technological Enterprise (CREATE) programme.
\bibliography{ref}

\newpage
\appendix
\section{Proof}
\subsection{Proof of \cref{thm:converge}}
\label{appendix1}
\begin{lemma}[Von Neumann~\cite{vonNeumann1949}] 
    \label{Lemma:von}
    Set $b_0 := x,a_n:=P_A(b_{n-1}), b_n:=P_B(a_n)$. We claim that both $\{a_1,a_2,a_3,...,a_n,...\}$ and $\{b_0,b_1,b_2,...,b_n,...\}$ will converge to $P_{A\cap B} (x)$ in norm whenever $A, B$ are closed subspaces.
    \end{lemma}
    \begin{proof}
    For $b_k = P_B(a_k)$, we have $(a_k-b_k)^T\cdot b_k=0.$ Then
    \begin{equation}
    \begin{split}
        \|a_k\|^2 = \|a_k-b_k+b_k\|^2\\
        =\|a_k-b_k\|^2+\|b_k\|^2+2(a_k-b_k)^T\cdot(b_k)\\
        =\|a_k-b_k\|^2+\|b_k\|^2.
    \end{split}
    \label{equ:akbk}
    \end{equation}
    Then $\|b_k\|^2\leq\|a_k\|^2.$\\
    Similarly, we could prove $\|a_{k+1}\|\leq \|b_k\|\leq \|a_{k}\|\leq \|b_{k-1}\|\leq ...\leq\|x\|.$
    Both $\{\|b_k\|\}$ and $\{\|a_k\|\}$ are non-increasing sequences, and converge to $B, A$. \\
    For a bounded sequence $\{a_k\}$, there exists a subsequence $\{a_{k_i}\}$ such that 
    \begin{equation}
        a_{k_i}\to \hat{x}\quad \mathrm{as}\ k_i\to +\infty.
    \end{equation}
    From~\eqref{equ:akbk}, we have $\|a_k-b_k\|\to 0$ as $k\to +\infty,$ then
    \begin{equation}
        \mathrm{dist}(a_{k_i},B)\leq \mathrm{dist}(a_{k_i},b_{k_i}) \to 0\quad \mathrm{as}\ k_i\to +\infty.
    \end{equation}
    Thus, $\hat{x}\in B$.
    Since $A$ is closed, $\hat{x}\in A\cap B.$\\
    For $b_k = P_B(a_k)$, we have $(a_k-b_k)^T\cdot (b_k-\hat{x})= 0.$ Then 
    \begin{equation}
        \begin{split}
        \|a_k-\hat{x}\|^2 = \|a_k-b_k+b_k-\hat{x}\|^2\\
        =\|a_k-b_k\|^2+\|b_k-\hat{x}\|^2+2(a_k-b_k)^T\cdot(b_k-\hat{x})\\
        =\|a_k-b_k\|^2+\|b_k-\hat{x}\|^2.
    \end{split}
    \end{equation}
    Then $\|a_k-\hat{x}\|\to 0$, $\|b_k-\hat{x}\|\to 0$ as $k\to +\infty.$\\
    We only need to illustrate $\hat{x}=P_{A\cap B}(x).$\\
    Separate $x$ into two terms
    \begin{equation}
        x = P_{A\cap B}(x) + (I-P_{A\cap B})(x),
    \end{equation}
    we observe that $P_A P_{A\cap B}=P_{A\cap B}P_B$, then when repeating the projection to $A$ and $B$, 
    \begin{equation}
        a_k = P_{A\cap B}(x) + (I-P_{A\cap B})(a_k)\to P_{A\cap B}(x)\quad \mathrm{as}\ k \to +\infty,
    \end{equation}
    that is, $\hat{x} = P_{A\cap B}(x).$
    \end{proof}

    \paragraph{Proof of Theorem~\ref{thm:converge}}
    \begin{proof}
        \begin{enumerate}[label = (\arabic*)]
            \item Let $n =1$,
            \begin{equation}
                \begin{split}
                    [(I-P_{\mathcal{G}})\circ (I-P_{\mathcal{H}})](F)& = (I-P_{\mathcal{G}}-P_{\mathcal{H}}+P_{\mathcal{G}}P_{\mathcal{H}})(F)\\
                    & = F - P_{\mathcal{G}}(I-P_{\mathcal{H}})(F) - P_{\mathcal{H}}(F)\\
                    & = F - P_{\mathcal{G}}(I-P_{\mathcal{H}})(F)+F_H^1\\
                    & = F - P_{\mathcal{G}}(F-F_H^1)-F_H^1\\
                    & = F-F_G^1 - F_H^1.
                \end{split}
            \end{equation}
            Suppose that when $n = k$, $[(I-P_{\mathcal{G}})\circ(I-P_{\mathcal{H}})]^k(F) = F-F_\mathcal{G}^k - F_\mathcal{H}^k$, then we have 
            \begin{equation}
                \begin{split}
                    [(I-P_{\mathcal{G}})\circ (I-P_{\mathcal{H}})](F-F_G^k - F_H^k)\\ 
                    = F-F_G^k-F_H^k - P_{\mathcal{G}}(I-P_{\mathcal{H}})(F-F_G^k - F_H^k) - P_{\mathcal{H}}(F-F_G^k)+F_H^k\\
                    = F-F_G^k-F_H^k - P_{\mathcal{G}}(F-F_G^k - F_H^{k+1}) - F_H^{k+1} + F_H^k\\
                    = F-F_G^{k+1} - F_H^{k+1}
                \end{split}
            \end{equation}
            Let $A = \mathcal{G}^{\bot}, B = \mathcal{H}^{\bot}, b_0 = F$ in Lemma~\ref{Lemma:von}, 
            \begin{equation}
                \begin{split}
                    \lim_{n\to +\infty}\|P_{\mathcal{G}\oplus \mathcal{H}}(F) - F_G^n - F_H^n\|\\=\lim_{n\to +\infty}\|(I - P_{\mathcal{G}\oplus \mathcal{H}})(F) - [(I-P_{\mathcal{G}})\circ(I-P_{\mathcal{H}})]^n(F)\| = 0.
                \end{split}
            \end{equation}
            \item Let $\{\tilde{x}_k\}$ be the trajectory generated by $\tilde{x}_{k+1} = P_{\mathcal{G}\oplus \mathcal{H}}(\hat{x}_k)$ and $\tilde{F} = P_{\mathcal{G}\oplus \mathcal{H}}(F)$.
            $\forall k$,
            \begin{equation}
                \begin{split}
                \epsilon_k = \|\tilde{x}_k - x_k^n\|&\leq \|\tilde{F}(\tilde{x}_{k-1}) - F^n(\tilde{x}_{k-1})\|+\|F^n(\tilde{x}_{k-1})-F^n(x_{k-1}^n)\|\\
                &\leq \|\tilde{F}-F^n\|\|\tilde{x}_{k-1}\| + \|F^n\|\epsilon_{k-1}
                \end{split}
            \end{equation}
            Since $\epsilon_0=0$, $\epsilon_1 \leq\|\tilde{F}-F^n\|\|\tilde{x}_0\|\to 0$. Let $\|\tilde{x}_{k-1}\| = M_{k-1}$. Then, $\forall \epsilon>0, \exists N_1>0$, such that 
            \begin{equation}
                \forall n>N_1\quad \|\tilde{F}-F^n\|\leq \frac{\epsilon}{2M_{k-1}}.
            \end{equation}
            $\exists N_2>0$, such that 
            \begin{equation}
                \epsilon_{k-1}<\epsilon/2(\frac{\epsilon}{2M_{k-1}+1+ \epsilon_F}).
            \end{equation}
            Thus, set $N = \max\{N_1, N_2\}$, we have
            \begin{equation}
                \epsilon_k < \frac{\epsilon}{2} + \frac{\epsilon}{2} <\epsilon,
            \end{equation}
            the trajectory $\{x^n_k\}$ generated by $x^n_{k+1} = F^n(x^n_k)$ with $x^n_0 = x_0$  will converge to a sequence $\{\tilde{x}_k\}$.\\
            For the distance from $\{\tilde{x}_k\}$ to $\{x_k\}$, we have
            \begin{equation}
                \begin{split}
                e_k = \|\tilde{x}_k - x_k\| & =\|\tilde{F}(\tilde{x}_{k-1})-\tilde{F}(x_{k-1})\|+\|\tilde{F}(x_{k-1})- F(x_{k-1})\| \\
                &\leq (1+\epsilon_F)e_{k-1} + \epsilon_F M.
                \end{split}
            \end{equation}
            Since $e_0=0$, 
            \begin{equation}
                e_k\leq ((1+\epsilon_F)^k-1)M.
            \end{equation}
        Specifically, if $F\in \mathcal{H}\oplus\mathcal{G}$, $\epsilon_F$ = 0, then $e_k = 0$, $e^n_k = \epsilon^n_k\to 0$ as $n\to +\infty$.
\end{enumerate}
    \end{proof}

\subsection{Proof of \cref{thm:convrate}}
\label{appendix2}
\begin{lemma}{~\cite{altprojmethods}}
    \label{lemma:c0}
    \begin{equation}
        c_0(A,B)=\|P_AP_B\|
    \end{equation}
    \end{lemma}
    
    \begin{proof}
    for $y = P_B(F)$, 
    \begin{equation}
        |\langle P_A(y), y\rangle| = \|P_A(y)\|\|P_AP_B(y)\| = \|P_A P_B(y)\|^2\leq c_0(A,B)\|P_A(y)\|\|y\|,
    \end{equation}
    then 
    \begin{equation}
        \|P_AP_B(y)\|\leq c_0(A,B)\|y\|.
    \end{equation}
    
    For $y\in B$, we have that 
    \begin{equation}
        P_A(y) = \mathrm{argmin}_{F\in A}\frac{|\langle F,y\rangle|}{\|F\|\|y\|}.
    \end{equation}
    For $A,B$ some closed subspaces, there exists $y_0\in B$ such that 
    \begin{equation}
        c(A,B) = \frac{|\langle P_A(y_0),y_0\rangle|}{\|P_A(y_0)\|\|y_0\|}.
    \end{equation}
    Thus, $\|P_AP_B\|=\inf\{c:\|P_AP_B(F)\|\leq c\|F\|,\forall F\} \geq c(A,B).$
    
    For any $F$, 
    \begin{equation}
        \|P_AP_B(F)\|=\|P_AP_B(P_B(F))\|\leq c(A,B) \|P_B(F)\|\leq c(A,B)\|F\|.
    \end{equation}
    Thus, $\|P_AP_B\|\leq c(A,B).$
    
    Therefore we have $\|P_AP_B\|=c(A,B)$.
    \end{proof}
    
 \paragraph{Proof of theorem \ref{thm:convrate}}
    \begin{proof}
        \begin{enumerate}[label = (\arabic*)]
        \item Let $A = \mathcal{G}^{\bot}, B = \mathcal{H}^{\bot}, b_0 = F$ in~\cref{lemma:c0}. We notice that $(P_A P_B)^n-P_{(A\cap B)} = (P_A P_B)^n(I-P_{(A\cap B)}) = (P_A P_B)^n P_{(A\cap B)^{\bot}}$ and $P_{(A\cap B)^{\bot}}$ commutes with both $P_A$ and $P_B$, we have
        \begin{equation}
            (P_A P_B)^n P_{(A\cap B)^{\bot}}=(P_A P_{(A\cap B)^{\bot}}P_B P_{(A\cap B)^{\bot}})^n
        \end{equation}
        and 
        \begin{equation}
            P_AP_{(A\cap B)^{\bot}} = P_{A\cap(A\cap B)^{\bot}},
        \end{equation}
        \begin{equation}
            P_BP_{(A\cap B)^{\bot}} = P_{B\cap(A\cap B)^{\bot}}.
        \end{equation}
    We introduce the notation $Q_A = P_A P_{(A\cap B)^{\bot}}, Q_B = P_B P_{(A\cap B)^{\bot}}$, 
    then we have 
    \begin{equation}
        (P_AP_B)^n-P_{(A\cap B)} =(P_AP_B)^nP_{(A\cap B)^{\bot}} = (Q_AQ_B)^n.
    \end{equation}
    Now, since 
    \begin{equation}
        \begin{split}
            \|(Q_A Q_B)^n\|^2 
            & = \|(Q_A Q_B)^n[(Q_A Q_B)^n]^*\|\\
            & = \|(Q_A Q_B)^n(Q_B Q_A)^n\|\quad \textrm{(projection\ operator\ is\ self-adjoint)}\\
            & = \|(Q_A Q_B Q_A)^{2n-1}\|\\
            & = \|Q_A Q_B Q_A\|^{2n-1}\quad (Q_A Q_B Q_A\ \textrm{is\ self-adjoint})\\
            & =\|Q_A Q_B Q_B Q_A\|^{2n-1}\\
            & =(\|Q_A Q_B\|^{2n-1})^2.
        \end{split}
        \label{equ:QAQB_norm}
    \end{equation}
    Then $\|(Q_AQ_B)^n\|=\|Q_AQ_B\|^{2n-1}.$ From Lemma \ref{lemma:c0}, we have $\|Q_AQ_B\|=c_0(A\cap(A\cap B)^{\bot}, B\cap(A\cap B)^{\bot}) = c(A,B).$
    Since
    \begin{equation}
    \begin{split}
        \|F^n - \tilde{F}\| & = \|(P_AP_B)^n(F)-P_{A\cap B}(F)\| \leq \|(P_AP_B)^n-P_{A\cap B}\|\|F\|\\
        & \leq c(A,B)^{2n-1}\|F\| = c(\mathcal{G}, \mathcal{H})^{2n-1}\|F\|.
    \end{split}
    \label{equ:A_27}
    \end{equation}
According to \cref{ass:1}, $\|F\| = 1$, we have
\begin{equation}
    \|F^n - \tilde{F}\|  \leq c(\mathcal{G}, \mathcal{H})^{2n-1}.
\end{equation}
\item For any finite $k\in \mathbb{Z}^+, n\in \mathbb{Z}^+$, the following inequality holds:
\begin{equation}
    \begin{split}
    e^n_k = \|x^n_k - x_k\| & =\|F^n(x^n_{k-1})-F^n(x_{k-1})\|+\|F^n(x_{k-1})-F(x_{k-1})\| \\
    & \leq \|F^n\|\cdot \|x_{k-1}^n-x_{k-1}\| + \|F^n - F\|\cdot \|x_{k-1}\|\\
    &\leq (1 + \|F^n - F\|)e_{k-1} + \|F^n - F\| M\\
    &\leq (1 + \epsilon_F + c(\mathcal{G}, \mathcal{H})^{2n-1})e_{k-1} + (\epsilon_F + c(\mathcal{G}, \mathcal{H})^{2n-1}) M,
    \end{split}
    \label{equ:proof_convrate_2}
\end{equation}
Since $e^n_k=0$, we have
\begin{equation}
    \begin{split}
    e_k^n = \|x^n_{k}-x_k\| & \leq [(1+\|F^n - F\|)^k-1]M\\
    & \leq [(1+c(\mathcal{G},\mathcal{H})^{2n-1} + \epsilon_F)^k-1]M.
    \end{split}
\end{equation}
\item Let $r^n = F^n - \tilde{F}$, we have
\begin{equation}
    r^{n+1} = Q_A Q_B(r^n)
\end{equation}
When $c(\mathcal{G},\mathcal{H}) = 1$, then \cref{alg:framework} will converge after one step, i.e, $r^{n+1} = r^n = 0,\ \forall n\geq 1$. We have $\|r^{n} - r^{n+1}\| = 0$. Thus, we only consider the case when $c(\mathcal{G},\mathcal{H}) < 1$.
\begin{equation}
    \begin{split}
        r^{n+1} & = \sum_{i = 1}^{+\infty}(r^{n+i} - r^{n+i + 1})\\
        & = \sum_{i = 1}^{+\infty}((Q_A Q_B)^i(r^n - r^{n+1})).
    \end{split}
\end{equation}
By~\eqref{equ:QAQB_norm}, we have
\begin{equation}
    \begin{split}
        \|r_{n+1}\|& \leq \sum_{i=1}^{ + \infty}(c(\mathcal{G},\mathcal{H})^{2n-1}\|r^n - r^{n+1}\|)\\
        & = \frac{c(\mathcal{G},\mathcal{H})}{1-c(\mathcal{G},\mathcal{H})^2} \|F^n - F^{n+1}\|.
    \end{split}
\end{equation}
Similar to~\eqref{equ:proof_convrate_2}, we have
\begin{equation}
    \begin{split}
    e^n_k = \|x^n_k - x_k\| & =\|F^n(x^n_{k-1})-F^n(x_{k-1})\|+\|F^n(x_k)-F(x_k)\| \\
    &\leq (1+\epsilon_F + \frac{c(\mathcal{G},\mathcal{H})}{1-c(\mathcal{G},\mathcal{H})^2} \|F^n - F^{n+1}\|)e_{k-1} + \epsilon_F M,
    \end{split}
\end{equation}
Since $e^n_k=0$, we have
\begin{equation}
    e_k^n = \|x^n_{k}-x_k\| \leq [(1+\frac{c(\mathcal{G},\mathcal{H})}{1-c(\mathcal{G},\mathcal{H})^2} \|F^n - F^{n+1}\| + \epsilon_F)^k-1]M.\\
\end{equation}

Specifically, if $F\in \mathcal{H}\oplus\mathcal{G}$, $\epsilon_F = 0$, the prediction sequence $x^n_{k+1}=F^n(x^n_k)$ has
\begin{equation}
    e_k^n = \|x^n_{k}-x_k\| \leq [(1+\frac{c(\mathcal{G},\mathcal{H})}{1-c(\mathcal{G},\mathcal{H})^2} \|F^n - F^{n+1}\|)^k-1]M.\\
\end{equation}

When $k\frac{c(\mathcal{G},\mathcal{H})}{1-c(\mathcal{G},\mathcal{H})^2} \|F^n - F^{n+1}\|\leq 1$, we have
\begin{equation}
\begin{split}
    e_k^n & \leq [(1+\frac{c(\mathcal{G},\mathcal{H})}{1-c(\mathcal{G},\mathcal{H})^2} \|F^n - F^{n+1}\|)^k-1]M \\
    & \leq [(1+k\frac{c(\mathcal{G},\mathcal{H})}{1-c(\mathcal{G},\mathcal{H})^2} \|F^n - F^{n+1}\| + \frac{k(k-1)}{2!}(\frac{c(\mathcal{G},\mathcal{H})}{1-c(\mathcal{G},\mathcal{H})^2} \|F^n - F^{n+1}\|)^2 + \cdots \\
    &+ (\frac{c(\mathcal{G},\mathcal{H})}{1-c(\mathcal{G},\mathcal{H})^2} \|F^n - F^{n+1}\|)^k)-1]M\\
\end{split}
\end{equation}
For $j\geq 2$, we have
\begin{equation}
    \begin{split}
    &\frac{k(k-1)(k-2)\cdots(k-j+1)}{j!} (\frac{c(\mathcal{G},\mathcal{H})}{1-c(\mathcal{G},\mathcal{H})^2}\|F^n - F^{n+1}\|)^j\\
    \leq & \frac{k^j}{j!}(\frac{c(\mathcal{G},\mathcal{H})}{1-c(\mathcal{G},\mathcal{H})^2}\|F^n - F^{n+1}\|)^j\\
    \leq & k\frac{c(\mathcal{G},\mathcal{H})}{1-c(\mathcal{G},\mathcal{H})^2}\|F^n - F^{n+1}\|.
    \end{split}
\end{equation}
Thus, we have
\begin{equation}
    \begin{split}
    e_k^n & \leq [1+k^2\frac{c(\mathcal{G},\mathcal{H})}{1-c(\mathcal{G},\mathcal{H})^2} \|F^n - F^{n+1}\|-1]M\\
        & = k^2M\frac{c(\mathcal{G},\mathcal{H})}{1-c(\mathcal{G},\mathcal{H})^2} \|F^n - F^{n+1}\|.
\end{split}
\end{equation}

\end{enumerate}
    \end{proof}
    
\subsection{Proof of acceleration scheme}
\label{appendix_acc}
Without loss of generalization, we suppose $\mathcal{H} = \mathrm{span}\{h\}$. 
Since $P_{\mathcal{G}^{\bot} \cap \mathcal{H}^{\bot}}(F) = 0$, we denote $F = F_\mathcal{G}+F_\mathcal{H}$. 
If any one of these three holds - $F_\mathcal{G} = 0$, $F_\mathcal{H} = 0$, or $\mathcal{G} = \mathcal{H}$ - \cref{alg:framework} could immediately converge to the optimal solution. 
Thus, we only consider the situation that $F_\mathcal{G} \neq 0$, $F_\mathcal{H} \neq 0$ and $\mathcal{G} \neq \mathcal{H}$. 
In this case,
\begin{equation}
\begin{split}
r^0 & = F - P_{\mathcal{G}}(F) = P_{\mathcal{G}^{\bot}}(F_\mathcal{G} + F_\mathcal{H})\\
&= P_{\mathcal{G}^{\bot}}(F_\mathcal{H}) = \lambda_1 P_{\mathcal{G}^{\bot}}(h),\qquad \mathrm{(Let\ F_\mathcal{H} = \lambda_1 h)}\\
r^1 & = P_{\mathcal{G}^{\bot}}\circ P_{\mathcal{H}^{\bot}}(r^0) = P_{\mathcal{G}^{\bot}}(F - P_{\mathcal{H}}(P_{\mathcal{G}^{\bot}}(r^0)))\\
& = P_{\mathcal{G}^{\bot}}(F) - P_{\mathcal{G}^{\bot}}(P_{\mathcal{H}}(P_{\mathcal{G}^{\bot}}(r^0)))\\
& = \lambda_1 P_{\mathcal{G}^{\bot}}(h) - \lambda_2 P_{\mathcal{G}^{\bot}}(h) \qquad (\mathrm{Let}\ P_{\mathcal{H}}(P_{\mathcal{G}^{\bot}}(r^0)) = \lambda_2 h)\\
& = \frac{\lambda_1-\lambda_2}{\lambda_1}r^0.
\end{split}
\end{equation}
Thus, $r^0$ and $r^1$ are linearly dependent.
\end{document}